\documentclass[12pt]{amsart}
\usepackage{amsfonts}
\usepackage{amscd}
\usepackage{amsmath,amssymb,amsthm}
\usepackage{color}
\usepackage{graphicx}
\usepackage[colorlinks=true,linkcolor=red,citecolor=blue]{hyperref}
\usepackage{subfigure,array,tabularx}

\newtheorem{definition}{Definition}[section]
\newtheorem{remark}{Remark}[section]

\newtheorem{theorem}{Theorem}[section]
\newtheorem{lemma}{Lemma}[section]
\newtheorem{corollary}{Corollary}[section]
\numberwithin{equation}{section}

\numberwithin{equation}{section}

\input{tcilatex}

\begin{document}
\title{On the Obata Theorem in a weighted Sasakian manifold}
\author{$^{\ast }$Shu-Cheng Chang$^{1}$}
\address{$^{1}$Department of Mathematics and Taida Institute for
Mathematical Sciences (TIMS), National \ Taiwan University, Taipei 10617,
Taiwan}
\email{scchang@math.ntu.edu.tw }
\author{$^{\ast \ast }$Daguang Chen$^{2}$}
\address{$^{2}$Department of Mathematical Sciences, Tsinghua University,
Beijing 100084, P.R. China}
\email{dgchen@tsinghua.edu.cn}
\author{$^{\ast }$Chin-Tung Wu$^{3}$}
\address{$^{3}$Department of applied Mathematics, National Pingtung
University, Pingtung 90003, Taiwan}
\email{ctwu@mail.nptu.edu.tw}
\thanks{$^{\ast }$Research supported in part by the MOST of Taiwan\\
$^{\ast \ast }$Research supported in part by NSFC grant No.
11831005/11571360.}

\begin{abstract}
In this paper, we generalize the CR Obata theorem to a compact strictly
pseudoconvex CR manifold with a weighted volume measure. More precisely, we
first derive the weighted CR Reilly's formula associated with the Witten
sub-Laplacian and obtain the corresponding first eigenvalue estimate. With
its applications, we obtain the CR\ Obata theorem in a compact weighted
Sasakian manifold with or without boundary.
\end{abstract}

\maketitle

\section{Introduction}

Let $(M,J,\theta ,d\sigma )$ be a compact strictly pseudoconvex CR $(2n+1)$%
-manifold with a weighted volume measure $d\sigma =e^{-\phi (x)}\theta
\wedge \left( d\theta \right) ^{n}$ for a given smooth weighted function $%
\phi $. In this paper, we first derive the weighted CR Reilly formula (\ref%
{0}) associated with the Witten sub-Laplacian (\cite{cckl})%
\begin{equation*}
\begin{array}{c}
\mathcal{L}=\Delta _{b}-\nabla _{b}\phi \cdot \nabla _{b}%
\end{array}%
\end{equation*}%
in a compact weighted strictly pseudoconvex CR $(2n+1)$-manifold with or
without boundary. Here $\nabla _{b}$ is the subgradient and $\Delta _{b}$ is
the sub-Laplacian as in section $2$. Secondly, we obtain the corresponding
first eigenvalue estimate for the Witten sub-Laplacian in a compact weighted
strictly pseudoconvex CR $(2n+1)$-manifold with or without smooth boundary.
With its applications, we obtain the CR\ Obata theorem in a compact weighted
Sasakian manifold with or without boundary which is served as a
generalization of results in \cite{cc1}, \cite{cc2} and \cite{lw}.

Note that the Witten sub-Laplacian $\mathcal{L}$ satisfies the following
integration by parts equation%
\begin{equation*}
\begin{array}{c}
\int_{M}g\left( \mathcal{L}f\right) d\sigma =-\int_{M}\langle \nabla
_{b}f,\nabla _{b}g\rangle d\sigma =\int_{M}f\left( \mathcal{L}g\right)
d\sigma ,%
\end{array}%
\end{equation*}%
for all smooth functions $f,$ $g$ in a compact weighted strictly
pseudoconvex CR $(2n+1)$-manifold $M$\ without boundary.\ As in \cite{cckl}
the ($\infty $-dimensional) Bakry-Emery pseudohermitian Ricci curvature $Ric(%
\mathcal{L})$ and the corresponding torsion $Tor(\mathcal{L})$ are defined by%
\begin{equation}
\begin{array}{l}
Ric(\mathcal{L})(W,W)=R_{\alpha \overline{\beta }}W^{\alpha }W^{\overline{%
\beta }}+(n+2){\func{Re}}[\phi _{\alpha \overline{\beta }}W^{\alpha }W^{%
\overline{\beta }}], \\ 
Tor(\mathcal{L})(W,W)=2{\func{Re}}[(\sqrt{-1}A_{\overline{\alpha }\overline{%
\beta }}-\frac{n+2}{n+1}\phi _{\overline{\alpha }\overline{\beta }})W^{%
\overline{\alpha }}W^{\overline{\beta }}],%
\end{array}
\label{2019}
\end{equation}%
for all $W=W^{\alpha }Z_{\alpha }+W^{\overline{\alpha }}Z_{\overline{\alpha }%
}\in T^{1,0}(M)\oplus T^{0,1}(M).$

Now we recall the weighted CR Paneitz operator $P_{0}^{\phi }$ (Definition %
\ref{d2})%
\begin{equation}
\begin{array}{c}
P_{\beta }^{\phi }f=P_{\beta }f-\frac{1}{2}\left\langle \nabla _{b}\phi
,\nabla _{b}f\right\rangle _{,\beta }+\frac{n}{2}\sqrt{-1}f_{0}\phi _{\beta
}.%
\end{array}
\label{40}
\end{equation}%
Here $P^{\phi }f=\sum_{\beta =1}^{n}(P_{\beta }^{\phi }f)\theta ^{\beta }$
and $\overline{P}^{\phi }f=\sum_{\beta =1}^{n}(\overline{P}_{\beta }^{\phi
}f)\theta ^{\overline{\beta }}$ (Definition \ref{d1}). Then the weighted CR
Paneitz operator $P_{0}^{\phi }$ is defined by%
\begin{equation*}
\begin{array}{c}
P_{0}^{\phi }f:=4e^{\phi }[\delta _{b}(e^{-\phi }P^{\phi }f)+\overline{%
\delta }_{b}(e^{-\phi }\overline{P}^{\phi }f)].%
\end{array}%
\end{equation*}

By using integrating by parts to the CR Bochner formula (\ref{11}) for $%
\mathcal{L}$ with respect to the given weighted volume measure $d\sigma $,
we derive the following weighted CR Reilly formula.

\begin{theorem}
\label{Reilly'sformula} Let $(M,J,\theta ,d\sigma )$ be a compact weighted
strictly pseudoconvex CR $(2n+1)$-manifold with or without boundary $\Sigma $%
. Then for any real smooth function $f$, we have%
\begin{equation}
\begin{array}{ll}
& \frac{n+1}{n}\int_{M}[(\mathcal{L}f)^{2}-\frac{2n}{n+1}\sum_{\beta ,\gamma
}|f_{\beta \gamma }-\frac{1}{2}f_{\beta }\phi _{\gamma }|^{2}]d\sigma \\ 
= & \frac{n+2}{4n}\int_{M}fP_{0}^{\phi }fd\sigma +\int_{M}[Ric(\mathcal{L})-%
\frac{n+1}{2}Tor(\mathcal{L})](\nabla _{b}f,\nabla _{b}f)d\sigma \\ 
& +\frac{n+1}{2}\int_{M}\left( \mathcal{L}f\right) \langle \nabla
_{b}f,\nabla _{b}\phi \rangle d\sigma -\frac{n+2}{4}\int_{M}[\mathcal{L}\phi
+\frac{1}{2(n+2)}|\nabla _{b}\phi |^{2}]|\nabla _{b}f|^{2}d\sigma \\ 
& +\frac{3}{4n}C_{n}\int_{\Sigma }(\mathcal{L}f)f_{e_{_{2n}}}d\Sigma
_{p}^{\phi }-\frac{n+2}{2n}\sqrt{-1}C_{n}\int_{\Sigma }f(P_{n}^{\phi }f-P_{%
\overline{n}}^{\phi }f)d\Sigma _{p}^{\phi } \\ 
& +C_{n}\int_{\Sigma }f_{e_{2n}}\Delta _{b}^{t}fd\Sigma _{p}^{\phi }+\frac{1%
}{2}\sqrt{-1}C_{n}\int_{\Sigma }(f^{\overline{\beta }}B_{n\overline{\beta }%
}f-f^{\beta }B_{\overline{n}\beta }f)d\Sigma _{p}^{\phi } \\ 
& +\frac{1}{4}C_{n}\int_{\Sigma }H_{p.h}f_{e_{2n}}^{2}d\Sigma _{p}^{\phi }-%
\frac{1}{2}C_{n}\int_{\Sigma }\alpha f_{e_{n}}f_{e_{2n}}d\Sigma _{p}^{\phi }+%
\frac{3}{4}C_{n}\int_{\Sigma }f_{0}f_{e_{n}}d\Sigma _{p}^{\phi } \\ 
& -\frac{n^{2}+3n-1}{4n}C_{n}\int_{\Sigma }\left\langle \nabla _{b}f,\nabla
_{b}\phi \right\rangle f_{e_{2n}}d\Sigma _{p}^{\phi }+\frac{n-1}{2n}%
C_{n}\int_{\Sigma }\left\langle \nabla _{b}f,\nabla _{b}\phi \right\rangle
_{e_{2n}}fd\Sigma _{p}^{\phi } \\ 
& +\frac{n+2}{8}C_{n}\int_{\Sigma }|\nabla _{b}f|^{2}\phi
_{e_{_{2n}}}d\Sigma _{p}^{\phi }+\frac{1}{4}C_{n}\int_{\Sigma
}\sum_{j,k=1}^{2n-1}\left\langle \nabla _{e_{j}}e_{2n},e_{k}\right\rangle
f_{e_{j}}f_{e_{k}}d\Sigma _{p}^{\phi } \\ 
& -\frac{1}{4}C_{n}\int_{\Sigma }\sum_{j=1}^{2n-1}[\left\langle \nabla
_{e_{j}}e_{n},e_{j}\right\rangle f_{e_{n}}+f_{e_{j}}\phi
_{e_{j}}]f_{e_{2n}}d\Sigma _{p}^{\phi }.%
\end{array}
\label{0}
\end{equation}%
Here $P_{0}^{\phi }$ is the weighted CR Paneitz operator on $M,\
C_{n}:=2^{n}n!;$ $B_{\beta \overline{\gamma }}f:=f_{\beta \overline{\gamma }%
}-\frac{1}{n}f_{\sigma }$ $^{\sigma }h_{\beta \overline{\gamma }},$ $\Delta
_{b}^{t}:=\frac{1}{2}\sum_{j=1}^{2n-1}[\left( e_{j}\right) ^{2}-(\nabla
_{e_{j}}e_{j})^{t}]$ is the tangential sub-Laplacian of $\Sigma $ and $%
H_{p.h}$ is the $p$-mean curvature of $\Sigma $ with respect to the
Legendrian normal $e_{2n},$ $\alpha e_{2n}+T\in T\Sigma $ for some function $%
\alpha $ on $\Sigma \backslash S_{\Sigma },$ the singular set $S_{\Sigma }$
consists of those points where the contact bundle $\xi =\ker \theta $
coincides with the tangent bundle $T\Sigma $ of $\Sigma ,$ and $d\Sigma
_{p}^{\phi }=e^{-\phi }\theta \wedge e^{1}\wedge e^{n+1}\wedge \cdots \wedge
e^{n-1}\wedge e^{2n-1}\wedge e^{n}$ is the weighted $p$-area element on $%
\Sigma .$
\end{theorem}

In Lemma \ref{lemma}, we observe that 
\begin{equation}
\begin{array}{lll}
P_{0}^{\phi }f & = & 2\left( \mathcal{L}^{2}+n^{2}T^{2}\right) f-4n{\func{Re}%
}Q^{\phi }f-2n^{2}\phi _{0}f_{0} \\ 
& = & 2\square _{b}^{\phi }\overline{\square }_{b}^{\phi }f-4nQ^{\phi
}f-2n^{2}\phi _{0}f_{0}-2n\sqrt{-1}\left\langle \nabla _{b}\phi _{0},\nabla
_{b}f\right\rangle \\ 
& = & 2\overline{\square }_{b}^{\phi }\square _{b}^{\phi }f-4n\overline{Q}%
^{\phi }f-2n^{2}\phi _{0}f_{0}+2n\sqrt{-1}\left\langle \nabla _{b}\phi
_{0},\nabla _{b}f\right\rangle ,%
\end{array}
\label{4c}
\end{equation}%
for the weighted Kohn Laplacian $\square _{b}^{\phi }f=(-\mathcal{L}+n\sqrt{%
-1}T)f.$ This implies that $P_{0}^{\phi }$ is a self-adjoint operator in a
compact weighted pseudohermitian $(2n+1)$-manifold without boundary.
However, in order for the CR Paneitz $P_{0}^{\phi }$ to be self-adjoint when 
$M$ with the nonempty smooth boundary $\Sigma ,$ one needs all smooth
functions on $M$ satisfy some\ suitable boundary conditions on $\Sigma $.
That is, one can consider the following Dirichlet eigenvalue problem for $%
P_{0}^{\phi }$: 
\begin{equation}
\begin{array}{c}
\left\{ 
\begin{array}{c}
P_{0}^{\phi }\varphi =\mu _{_{D}}\varphi \ \ \mathrm{on\ }M, \\ 
\varphi =0=\mathcal{L}\varphi \ \mathrm{on\ }\Sigma .%
\end{array}%
\right.%
\end{array}
\label{1}
\end{equation}%
Hence%
\begin{equation}
\begin{array}{c}
\int_{M}\varphi P_{0}^{\phi }\varphi d\sigma \geq \mu
_{_{D}}^{1}\int_{M}\varphi ^{2}d\sigma%
\end{array}
\label{2}
\end{equation}%
for the first Dirichlet eigenvalue $\mu _{_{D}}^{1}$ and all smooth
functions on $M$ with $\varphi =0=\mathcal{L}\varphi $ on $\Sigma .$ We
refer \cite{ccw} for some details in case that $\phi $ is constant.

In general, $\mu _{_{D}}^{1}$ is not always nonnegative. It is related to
the nonnegativity of the weighted CR Paneitz operator $P_{0}^{\phi }$ in a
compact weighted strictly pseudoconvex CR $(2n+1)$-manifold.

\begin{definition}
Let $(M,J,\theta ,d\sigma )$ be a compact weighted strictly pseudoconvex CR $%
(2n+1)$-manifold with smooth boundary $\Sigma $. We say that the weighted CR
Paneitz operator $P_{0}^{\phi }$ is nonnegative if 
\begin{equation*}
\begin{array}{c}
\int_{M}\varphi P_{0}^{\phi }\varphi d\mu \geq 0%
\end{array}%
\end{equation*}%
for all smooth functions $\varphi $ with suitable boundary conditions (\ref%
{1}) as in Dirichlet eigenvalue problem.
\end{definition}

\begin{remark}
\label{r1} Let $(M,J,\theta ,d\sigma )$ be a compact weighted strictly
pseudoconvex CR $(2n+1)$-manifold of vanishing torsion with or without
smooth boundary $\Sigma $ and $\phi _{0}$ vanishes on $M$. It follows from (%
\ref{4c}) that the weighted Kohn Laplacian $\square _{b}^{\phi }$ and $%
\overline{\square }_{b}^{\phi }$ commute and they are diagonalized
simultaneously with 
\begin{equation*}
\begin{array}{c}
P_{0}^{\phi }=\square _{b}^{\phi }\overline{\square }_{b}^{\phi }+\overline{%
\square }_{b}^{\phi }\square _{b}^{\phi }.%
\end{array}%
\end{equation*}%
Then the corresponding weighted CR Paneitz operator $P_{0}^{\phi }$ is
nonnegative (Lemma \ref{l41}). That is $\mu _{_{D}}^{1}\geq 0.$
\end{remark}

With its applications, we first derive the first eigenvalue estimate and
weighted CR\ Obata theorem in a closed weighted strictly pseudoconvex CR $%
(2n+1)$-manifold.

\begin{theorem}
\label{Thm}Let $(M,J,\theta ,d\sigma )$ be a closed weighted strictly
pseudoconvex CR $(2n+1)$-manifold with the nonnegative weighted CR Paneitz
operator $P_{0}^{\phi }$. Suppose that 
\begin{equation}
\begin{array}{c}
\lbrack Ric(\mathcal{L})-\frac{n+1}{2}Tor(\mathcal{L})](Z,Z)\geq
k\left\langle Z,Z\right\rangle%
\end{array}
\label{2019B}
\end{equation}%
for all $Z\in T_{1,0}$ and a positive constant $k.$ Then the first
eigenvalue of the Witten sub-Laplacian $\mathcal{L}$ satisfies the lower
bound 
\begin{equation}
\begin{array}{c}
\lambda _{1}\geq \frac{2n[k-(n+2)l]}{(n+1)(2+n\omega )},%
\end{array}
\label{01a}
\end{equation}%
where $\omega =\underset{M}{\mathrm{osc}}\phi =\underset{M}{\sup }\phi -%
\underset{M}{\inf }\phi $ and for nonnegative constant $l$ with\ $0\leq l<%
\frac{k}{n+2}$ such that%
\begin{equation}
\begin{array}{c}
\mathcal{L}\phi +\frac{1}{2(n+2)}|\nabla _{b}\phi |^{2}\leq 4l%
\end{array}
\label{2019A}
\end{equation}%
on $M.$ Moreover, if the equality (\ref{01a}) holds, then $M$\ is CR
isometric to a standard CR $(2n+1)$-sphere.
\end{theorem}

Furthermore, $P_{0}^{\phi }$ is nonnegative if the torsion is zero (i.e.
Sasakian) and $\phi _{0}$ vanishes (Lemma \ref{l41}). Then we have the
following CR\ Obata theorem in a closed weighted Sasakian $(2n+1)$-manifold.

\begin{corollary}
\label{C1} Let $(M,J,\theta ,d\sigma )$ be a closed weighted Sasakian $%
(2n+1) $-manifold. Suppose that 
\begin{equation*}
\begin{array}{c}
\lbrack Ric(\mathcal{L})-\frac{n+1}{2}Tor(\mathcal{L})](Z,Z)\geq
k\left\langle Z,Z\right\rangle%
\end{array}%
\end{equation*}%
and 
\begin{equation*}
\phi _{0}=0,
\end{equation*}%
for all $Z\in T_{1,0},$ a positive constant $k.$ Then the first eigenvalue
of the Witten sub-Laplacian $\mathcal{L}$ satisfies the lower bound 
\begin{equation}
\begin{array}{c}
\lambda _{1}\geq \frac{2n[k-(n+2)l]}{(n+1)(2+n\omega )},%
\end{array}
\label{01b}
\end{equation}%
where $\omega =\underset{M}{\mathrm{osc}}\phi =\underset{M}{\sup }\phi -%
\underset{M}{\inf }\phi $ and for nonnegative constant $l$ with\ $0\leq l<%
\frac{k}{n+2}$ such that 
\begin{equation*}
\begin{array}{c}
\mathcal{L}\phi +\frac{1}{2(n+2)}|\nabla _{b}\phi |^{2}\leq 4l%
\end{array}%
\end{equation*}%
on $M.$ Moreover, if the equality (\ref{01b}) holds, then $M$\ is CR
isometric to a standard CR $(2n+1)$-sphere.
\end{corollary}

\begin{remark}
1. Theorem \ref{Thm} and Corollary \ref{C1} are done as in \cite{gr}, \cite%
{ch} and \cite{cc3} in case that the weighted function $\phi $ is constant
in which $l=0$.

2. Note that (\ref{2019A}) is equivalent to 
\begin{equation*}
\begin{array}{c}
\mathcal{L}\phi +\frac{1}{2(n+2)}|\nabla _{b}\phi |^{2}=\Delta _{b}\phi -%
\frac{2n+3}{2(n+2)}|\nabla _{b}\phi |^{2}\leq 4l<\frac{4k}{n+2}.%
\end{array}%
\end{equation*}%
Then by comparing (\ref{2019}), (\ref{2019A}) and (\ref{2019B}), it has a
plenty of rooms for the choice of the weighted function $\phi $. For
example, it is the case by a small perturbation of the subhessian of the
weighted function $\phi .$
\end{remark}

Secondly, we consider the following Dirichlet eigenvalue problem of the
Witten sub-Laplacian $\mathcal{L}$ in a compact weighted strictly
pseudoconvex CR $(2n+1)$-manifold $M$ with smooth boundary $\Sigma $: 
\begin{equation}
\left\{ 
\begin{array}{ccll}
\mathcal{L}f & = & -\lambda _{1}f & \mathrm{on\ }M, \\ 
f & = & 0 & \mathrm{on\ }\Sigma .%
\end{array}%
\right.  \label{1b}
\end{equation}%
Then we have the following CR first\ Dirichlet eigenvalue estimate and its
weighted Obata Theorem.

\begin{theorem}
\label{TB} Let $(M,J,\theta ,d\sigma )$ be a compact weighted strictly
pseudoconvex CR $(2n+1)$-manifold with the smooth boundary $\Sigma $ and the
weighted CR Paneitz operator $P_{0}^{\phi }$ is nonnegative. Suppose that 
\begin{equation*}
\begin{array}{c}
\lbrack Ric(\mathcal{L})-\frac{n+1}{2}Tor(\mathcal{L})](Z,Z)\geq
k\left\langle Z,Z\right\rangle%
\end{array}%
\end{equation*}%
for all $Z\in T_{1,0},$ and the pseudohermitian mean curvature and
connection $1$-form satisfies 
\begin{equation*}
\begin{array}{c}
H_{p.h}-\tilde{\omega}_{n}^{\;n}(e_{n})-\frac{n+2}{2}\phi _{e_{2n}}\geq 0%
\end{array}%
\end{equation*}%
on $\Sigma $ and $H_{p.h}+\tilde{\omega}_{n}^{\;n}(e_{n})$ is also
nonnegative on $\Sigma $ for $n\geq 2$. Then the first Dirichlet eigenvalue
of the Witten sub-Laplacian $\mathcal{L}$ satisfies the lower bound 
\begin{equation}
\begin{array}{c}
\lambda _{1}\geq \frac{2n[k-(n+2)l]}{(n+1)(2+n\omega )},%
\end{array}
\label{01c}
\end{equation}%
where $\omega =\underset{M}{\mathrm{osc}}\phi =\underset{M}{\sup }\phi -%
\underset{M}{\inf }\phi $ and for nonnegative constant $l$ with\ $0\leq l<%
\frac{k}{n+2}$ such that%
\begin{equation*}
\begin{array}{c}
\mathcal{L}\phi +\frac{1}{2(n+2)}|\nabla _{b}\phi |^{2}\leq 4l%
\end{array}%
\end{equation*}%
on $M.\ $Moreover, if the equality (\ref{01c}) holds, then $M$\ is isometric
to a hemisphere in a standard CR $(2n+1)$-sphere.
\end{theorem}

\begin{corollary}
Let $(M,J,\theta ,d\sigma )$ be a compact weighted Sasakian $(2n+1)$%
-manifold with smooth boundary $\Sigma $. Suppose that 
\begin{equation*}
\begin{array}{c}
\lbrack Ric(\mathcal{L})-\frac{n+1}{2}Tor(\mathcal{L})](Z,Z)\geq
k\left\langle Z,Z\right\rangle%
\end{array}%
\end{equation*}%
for all $Z\in T_{1,0},$ and 
\begin{equation*}
\phi _{0}=0.
\end{equation*}%
Furthermore, assume that the pseudohermitian mean curvature and connection $%
1 $-form satisfies%
\begin{equation*}
\begin{array}{c}
H_{p.h}-\tilde{\omega}_{n}^{\;n}(e_{n})\geq 0%
\end{array}%
\end{equation*}%
on $\Sigma $ and $H_{p.h}+\tilde{\omega}_{n}^{\;n}(e_{n})$ is also
nonnegative on $\Sigma $ if $n\geq 2.$ Then the first Dirichlet eigenvalue
of the Witten sub-Laplacian $\mathcal{L}$ satisfies the lower bound 
\begin{equation*}
\begin{array}{c}
\lambda _{1}\geq \frac{2n[k-(n+2)l]}{(n+1)(2+n\omega )},%
\end{array}%
\end{equation*}%
where $\omega =\underset{M}{\mathrm{osc}}\phi =\underset{M}{\sup }\phi -%
\underset{M}{\inf }\phi $ and for nonnegative constant $l$ with\ $0\leq l<%
\frac{k}{n+2}$ such that%
\begin{equation*}
\begin{array}{c}
\mathcal{L}\phi +\frac{1}{2(n+2)}|\nabla _{b}\phi |^{2}\leq 4l%
\end{array}%
\end{equation*}%
on $M.\ $Moreover, if the equality holds then $M$\ is isometric to a
hemisphere in a standard CR $(2n+1)$-sphere.
\end{corollary}

We briefly describe the methods used in our proofs. In section $2$, we
introduce the weighted CR Paneitz operator $P_{0}^{\phi }$. In section $3$,
by using integrating by parts to the weighted CR Bochner formula (\ref{11}),
we can derive the CR version of weighted Reilly's formula. By applying the
weighted CR Reilly's formula, we are able to obtain the first eigenvalue
estimate of the Witten sub-Laplacian as in section $4$ in a closed weighted
strictly pseudoconvex CR $(2n+1)$-manifold and its weighted Obata Theorem.
In section $5,$ we derive the first Dirichlet eigenvalue estimate in a
compact weighted strictly pseudoconvex CR $(2n+1)$-manifold with boundary $%
\Sigma $ and its corresponding weighted Obata-type Theorem.

\textbf{Acknowledgements} This work was partially done while the second
author visited Taida Institute of Mathematical Sciences (TIMS), Taiwan. He
would like to thank the institute for its hospitality.

\section{The weighted CR Paneitz Operator}

We first introduce some basic materials in a strictly pseudoconvex CR $%
(2n+1) $-manifold $(M,J,\theta )$. Let $(M,J,\theta )$ be a $(2n+1)$%
-dimensional, orientable, contact manifold with contact structure $\xi =\ker
\theta $. A CR structure compatible with $\xi $ is an endomorphism $J:\xi
\rightarrow \xi $ such that $J^{2}=-1$. We also assume that $J$ satisfies
the following integrability condition: If $X$ and $Y$ are in $\xi $, then so
is $[JX,Y]+[X,JY]$ and $J([JX,Y]+[X,JY])=[JX,JY]-[X,Y]$. A CR structure $J$
can extend to $\mathbb{C}\mathbf{\otimes }\xi $ and decomposes $\mathbb{C}%
\mathbf{\otimes }\xi $ into the direct sum of $T_{1,0}$ and $T_{0,1}$ which
are eigenspaces of $J$ with respect to eigenvalues $\sqrt{-1}$ and $-\sqrt{-1%
}$, respectively. A manifold $M$ with a CR structure is called a CR
manifold. A pseudohermitian structure compatible with $\xi $ is a $CR$
structure $J$ compatible with $\xi $ together with a choice of contact form $%
\theta $. Such a choice determines a unique real vector field $T$ transverse
to $\xi $, which is called the characteristic vector field of $\theta $,
such that ${\theta }(T)=1$ and $\mathcal{L}_{T}{\theta }=0$ or $d{\theta }(T,%
{\cdot })=0$. Let $\left\{ T,Z_{\beta },Z_{\overline{\beta }}\right\} $ be a
frame of $TM\otimes \mathbb{C}$, where $Z_{\beta }$ is any local frame of $%
T_{1,0},\ Z_{\overline{\beta }}=\overline{Z_{\beta }}\in T_{0,1}$ and $T$ is
the characteristic vector field. Then $\{\theta ,\theta ^{\beta },\theta ^{%
\overline{\beta }}\}$, which is the coframe dual to $\left\{ T,Z_{\beta },Z_{%
\overline{\beta }}\right\} $, satisfies 
\begin{equation}
\begin{array}{c}
d\theta =\sqrt{-1}h_{\beta \overline{\gamma }}\theta ^{\beta }\wedge \theta
^{\overline{\gamma }},%
\end{array}
\label{dtheta}
\end{equation}%
for some positive definite Hermitian matrix of functions $(h_{\beta 
\overline{\gamma }})$. Actually we can always choose $Z_{\beta }$ such that $%
h_{\beta \overline{\gamma }}=\delta _{\beta \gamma }$; hence, throughout
this note, we assume $h_{\beta \overline{\gamma }}=\delta _{\beta \gamma }$.

The Levi form $\left\langle \ ,\ \right\rangle $ is the Hermitian form on $%
T_{1,0}$ defined by%
\begin{equation*}
\begin{array}{c}
\left\langle Z,W\right\rangle =-\sqrt{-1}\left\langle d\theta ,Z\wedge 
\overline{W}\right\rangle .%
\end{array}%
\end{equation*}%
We can extend $\left\langle \ ,\ \right\rangle $ to $T_{0,1}$ by defining $%
\left\langle \overline{Z},\overline{W}\right\rangle =\overline{\left\langle
Z,W\right\rangle }$ for all $Z,W\in T_{1,0}$. The Levi form induces
naturally a Hermitian form on the dual bundle of $T_{1,0}$, also denoted by $%
\left\langle \ ,\ \right\rangle $, and hence on all the induced tensor
bundles. Integrating the Hermitian form (when acting on sections) over $M$
with respect to the volume form $d\mu =\theta \wedge (d\theta )^{n}$, we get
an inner product on the space of sections of each tensor bundle.

The pseudohermitian connection of $(J,\theta )$ is the connection $\nabla $
on $TM\otimes \mathbb{C}$ (and extended to tensors) given in terms of a
local frame $Z_{\beta }\in T_{1,0}$ by%
\begin{equation*}
\nabla Z_{\beta }=\theta _{\beta }{}^{\gamma }\otimes Z_{\gamma },\quad
\nabla Z_{\overline{\beta }}=\theta _{\overline{\beta }}{}^{\overline{\gamma 
}}\otimes Z_{\overline{\gamma }},\quad \nabla T=0,
\end{equation*}%
where $\theta _{\beta }{}^{\gamma }$ are the $1$-forms uniquely determined
by the following equations:%
\begin{equation}
\begin{split}
d\theta ^{\beta }& =\theta ^{\gamma }\wedge \theta _{\gamma }{}^{\beta
}+\theta \wedge \tau ^{\beta }, \\
\tau _{\beta }\wedge \theta ^{\beta }& =0=\theta _{\beta }{}^{\gamma
}+\theta _{\overline{\beta }}{}^{\overline{\gamma }}.
\end{split}
\label{structure equs}
\end{equation}%
We can write (by Cartan lemma) $\tau _{\beta }=A_{\beta \gamma }\theta
^{\gamma }$ with $A_{\beta \gamma }=A_{\gamma \beta }$. The curvature of the
Tanaka-Webster connection, expressed in terms of the coframe $\{\theta
=\theta ^{0},\theta ^{\beta },\theta ^{\overline{\beta }}\}$, is 
\begin{equation*}
\begin{split}
\Pi _{\beta }{}^{\gamma }& =\overline{\Pi _{\bar{\beta}}{}^{\overline{\gamma 
}}}=d\theta _{\beta }{}^{\gamma }-\theta _{\beta }{}^{\sigma }\wedge \theta
_{\sigma }{}^{\gamma }, \\
\Pi _{0}{}^{\beta }& =\Pi _{\beta }{}^{0}=\Pi _{0}{}^{\bar{\beta}}=\Pi _{%
\bar{\beta}}{}^{0}=\Pi _{0}{}^{0}=0.
\end{split}%
\end{equation*}%
Webster showed that $\Pi _{\beta }{}^{\gamma }$ can be written 
\begin{equation*}
\begin{array}{c}
\Pi _{\beta }{}^{\gamma }=R_{\beta }{}^{\gamma }{}_{\rho \bar{\sigma}}\theta
^{\rho }\wedge \theta ^{\bar{\sigma}}+W_{\beta }{}^{\gamma }{}_{\rho }\theta
^{\rho }\wedge \theta -W^{\gamma }{}_{\beta \bar{\rho}}\theta ^{\bar{\rho}%
}\wedge \theta +\sqrt{-1}(\theta _{\beta }\wedge \tau ^{\gamma }-\tau
_{\beta }\wedge \theta ^{\gamma })%
\end{array}%
\end{equation*}%
where the coefficients satisfy 
\begin{equation*}
\begin{array}{c}
R_{\beta \overline{\gamma }\rho \bar{\sigma}}=\overline{R_{\gamma \bar{\beta}%
\sigma \bar{\rho}}}=R_{\overline{\gamma }\beta \bar{\sigma}\rho }=R_{\rho 
\overline{\gamma }\beta \bar{\sigma}},\ \ W_{\beta \overline{\gamma }\rho
}=W_{\rho \overline{\gamma }\beta }.%
\end{array}%
\end{equation*}

We will denote components of covariant derivatives with indices preceded by
comma; thus write $A_{\rho \beta ,\gamma }$. The indices $\{0,\beta ,%
\overline{\beta }\}$ indicate derivatives with respect to $\{T,Z_{\beta },Z_{%
\overline{\beta }}\}$. For derivatives of a scalar function, we will often
omit the comma, for instance, $u_{\beta }=Z_{\beta }u,\ u_{\gamma \bar{\beta}%
}=Z_{\bar{\beta}}Z_{\gamma }u-\theta _{\gamma }{}^{\rho }(Z_{\bar{\beta}%
})Z_{\rho }u,\ u_{0}=Tu$ for a smooth function $u$ .

For a real function $u$, the subgradient $\nabla _{b}$ is defined by $\nabla
_{b}u\in \xi $ and $\left\langle Z,\nabla _{b}u\right\rangle =du(Z)$ for all
vector fields $Z$ tangent to contact plane. Locally $\nabla _{b}u=u^{\beta
}Z_{\beta }+u^{\overline{\beta }}Z_{\overline{\beta }}$. We can use the
connection to define the subhessian as the complex linear map 
\begin{equation*}
\begin{array}{c}
(\nabla ^{H})^{2}u:T_{1,0}\oplus T_{0,1}\rightarrow T_{1,0}\oplus T_{0,1}%
\text{\ \textrm{by}\ }(\nabla ^{H})^{2}u(Z)=\nabla _{Z}\nabla _{b}u.%
\end{array}%
\end{equation*}%
In particular, 
\begin{equation*}
\begin{array}{c}
|\nabla _{b}u|^{2}=2\sum_{\beta }u_{\beta }u^{\beta },\quad |\nabla
_{b}^{2}u|^{2}=2\sum_{\beta ,\gamma }(u_{\beta \gamma }u^{\beta \gamma
}+u_{\beta \overline{\gamma }}u^{\beta \overline{\gamma }}).%
\end{array}%
\end{equation*}%
Also the sub-Laplacian is defined by 
\begin{equation*}
\begin{array}{c}
\Delta _{b}u=Tr\left( (\nabla ^{H})^{2}u\right) =\sum_{\beta }(u_{\beta
}{}^{\beta }+u_{\overline{\beta }}{}^{\overline{\beta }}).%
\end{array}%
\end{equation*}%
The pseudohermitian Ricci tensor and the torsion tensor on $T_{1,0}$ are
defined by 
\begin{equation*}
\begin{array}{l}
Ric(X,Y)=R_{\gamma \bar{\beta}}X^{\gamma }Y^{\bar{\beta}} \\ 
Tor(X,Y)=\sqrt{-1}\sum_{\gamma ,\beta }(A_{\overline{\gamma }\bar{\beta}}X^{%
\overline{\gamma }}Y^{\bar{\beta}}-A_{\gamma \beta }X^{\gamma }Y^{\beta }),%
\end{array}%
\end{equation*}%
where $X=X^{\gamma }Z_{\gamma },\ Y=Y^{\beta }Z_{\beta }$.

Let $M$ be a compact strictly pseudoconvex CR $(2n+1)$-manifold with a
weighted volume measure $d\sigma =e^{-\phi (x)}d\mu $ for a given smooth
function $\phi $. In this section, we define the weighted CR Paneitz
operator $P_{0}^{\phi }$. First we recall the definition of the CR Paneitz
operator $P_{0}.$

\begin{definition}
\label{d1} (\cite{gl}) Let $(M,J,\theta )$ be a compact strictly
pseudoconvex CR $(2n+1)$-manifold. We define 
\begin{equation*}
\begin{array}{c}
Pf=\sum_{\gamma ,\beta =1}^{n}(f_{\overline{\gamma }\;\ \beta }^{\,\overline{%
\text{ }\gamma }}+\sqrt{-1}nA_{\beta \gamma }f^{\gamma })\theta ^{\beta
}=\sum_{\beta =1}^{n}(P_{\beta }f)\theta ^{\beta },%
\end{array}%
\end{equation*}%
which is an operator that characterizes CR-pluriharmonic functions. Here 
\begin{equation*}
\begin{array}{c}
P_{\beta }f=\sum_{\gamma =1}^{n}(f_{\overline{\gamma }\;\ \beta }^{\,\text{\ 
}\overline{\gamma }}+\sqrt{-1}nA_{\beta \gamma }f^{\gamma }),\text{ \ }\beta
=1,\cdots ,n,%
\end{array}%
\end{equation*}%
and $\overline{P}f=\sum_{\beta =1}^{n}\left( \overline{P}_{\beta }f\right)
\theta ^{\overline{\beta }}$, the conjugate of $P$. The CR Paneitz operator $%
P_{0}$ is defined by 
\begin{equation*}
\begin{array}{c}
P_{0}f=4[\delta _{b}(Pf)+\overline{\delta }_{b}(\overline{P}f)],%
\end{array}%
\end{equation*}%
where $\delta _{b}$ is the divergence operator that takes $(1,0)$-forms to
functions by $\delta _{b}(\sigma _{\beta }\theta ^{\beta })=\sigma _{\beta
}^{\;\ \beta }$, and similarly, $\overline{\delta }_{b}(\sigma _{\overline{%
\beta }}\theta ^{\overline{\beta }})=\sigma _{\overline{\beta }}^{\;\ 
\overline{\beta }}$.
\end{definition}

One can define (\cite{gl}) the purely holomorphic second-order operator $Q$
by 
\begin{equation*}
\begin{array}{c}
Qf:=2\sqrt{-1}(A^{\alpha \beta }f_{\alpha }),_{\beta }.%
\end{array}%
\end{equation*}%
Note that $[\Delta _{b},T]f=2{\func{Im}}Qf$ and observe that 
\begin{equation}
\begin{array}{lll}
P_{0}f & = & 2(\Delta _{b}^{2}+n^{2}T^{2})f-4n\mathrm{Re}Qf \\ 
& = & 2\square _{b}\overline{\square }_{b}f-4nQf\text{ }=\text{ }2\overline{%
\square }_{b}\square _{b}f-4n\overline{Q}f,%
\end{array}
\label{3}
\end{equation}%
for $\square _{b}f=(-\Delta _{b}+n\sqrt{-1}T)f=-2f_{\overline{\beta }}$ $^{%
\overline{\beta }}$ be the Kohn Laplacian operator.

With respect to the weighted volume measure $d\sigma ,$ we define the purely
holomorphic second-order operator $Q^{\phi }$ by%
\begin{equation*}
\begin{array}{c}
Q^{\phi }f:=Qf-2\sqrt{-1}(A^{\alpha \beta }f_{\alpha }\phi _{\beta })=2\sqrt{%
-1}e^{\phi }(e^{-\phi }A^{\alpha \beta }f_{\alpha }),_{\beta },%
\end{array}%
\end{equation*}%
and thus we have%
\begin{equation}
\begin{array}{c}
\lbrack \mathcal{L},T]f=2{\func{Im}}Q^{\phi }f+\left\langle \nabla _{b}\phi
_{0},\nabla _{b}f\right\rangle .%
\end{array}
\label{3a}
\end{equation}

\begin{definition}
\label{d2} We define the weighted CR Paneitz operator $P_{0}^{\phi }$ as
follows%
\begin{equation*}
\begin{array}{c}
P_{0}^{\phi }f=4e^{\phi }[\delta _{b}(e^{-\phi }P^{\phi }f)+\overline{\delta 
}_{b}(e^{-\phi }\overline{P}^{\phi }f)].%
\end{array}%
\end{equation*}%
Here $P^{\phi }f=\sum_{\beta =1}^{n}(P_{\beta }^{\phi }f)\theta ^{\beta }$
and $\overline{P}^{\phi }f=\sum_{\beta =1}^{n}(\overline{P}_{\beta }^{\phi
}f)\theta ^{\overline{\beta }}$, the conjugate of $P^{\phi }$ with 
\begin{equation*}
\begin{array}{c}
P_{\beta }^{\phi }f=P_{\beta }f-\frac{1}{2}\left\langle \nabla _{b}\phi
,\nabla _{b}f\right\rangle _{,\beta }+\frac{n}{2}\sqrt{-1}f_{0}\phi _{\beta
}.%
\end{array}%
\end{equation*}
\end{definition}

We explain why the weighted CR Paneitz operator $P_{0}^{\phi }$ to be
defined in this way. Comparing with the Riemannian case, we have the extra
term $\langle J\nabla _{b}f,\nabla _{b}f_{0}\rangle $ in the CR Bochner
formula (\ref{A}) for $\Delta _{b}$, which is hard to deal. From \cite{cc2},
we can relate $\langle J\nabla _{b}f,\nabla _{b}f_{0}\rangle $ with $\langle
\nabla _{b}f,\nabla _{b}\Delta _{b}f\rangle $ by%
\begin{equation*}
\begin{array}{c}
\langle J\nabla _{b}f,\nabla _{b}f_{0}\rangle =\frac{1}{n}\langle \nabla
_{b}f,\nabla _{b}\Delta _{b}f\rangle -\frac{2}{n}\langle Pf+\overline{P}%
f,d_{b}f\rangle -2Tor(\nabla _{b}f_{\mathbb{C}},\nabla _{b}f_{\mathbb{C}}),%
\end{array}%
\end{equation*}%
then by integral with respect to the volume measure $d\mu =\theta \wedge
\left( d\theta \right) ^{n}$ yields%
\begin{equation}
\begin{array}{c}
n^{2}\int_{M}f_{0}^{2}d\mu =\int_{M}\left( \Delta _{b}f\right) ^{2}d\mu -%
\frac{1}{2}\int_{M}fP_{0}fd\mu +n\int_{M}Tor(\nabla _{b}f,\nabla _{b}f)d\mu .%
\end{array}
\label{5}
\end{equation}%
This integral says that the integral of the square of $f_{0}$ can be replace
by the integral of the square of $\Delta _{b}f$ and the integral of the CR
Paneitz operator $P_{0}.$ For the CR Bochner formula (\ref{Bochnerformula})
for $\mathcal{L}$, we also have the extra term $\langle J\nabla _{b}f,\nabla
_{b}f_{0}\rangle -f_{0}\langle J\nabla _{b}f,\nabla _{b}\phi \rangle $,
which can be related by%
\begin{equation*}
\begin{array}{c}
\langle J\nabla _{b}f,\nabla _{b}f_{0}\rangle -f_{0}\langle J\nabla
_{b}f,\nabla _{b}\phi \rangle =\frac{1}{n}\langle \nabla _{b}f,\nabla _{b}%
\mathcal{L}f\rangle -\frac{2}{n}\langle P^{\phi }f+\overline{P}^{\phi
}f,d_{b}f\rangle -Tor(\nabla _{b}f,\nabla _{b}f).%
\end{array}%
\end{equation*}%
Then by integral with respect to the weighted volume measure $d\sigma
=e^{-\phi }\theta \wedge \left( d\theta \right) ^{n},$ one gets%
\begin{equation*}
\begin{array}{c}
n^{2}\int_{M}f_{0}^{2}d\sigma =\int_{M}\left( \mathcal{L}f\right)
^{2}d\sigma -\frac{1}{2}\int_{M}fP_{0}^{\phi }fd\sigma +n\int_{M}Tor(\nabla
_{b}f,\nabla _{b}f)d\sigma .%
\end{array}%
\end{equation*}%
This integral have the same type as (\ref{5}) when we replace $\mathcal{L},$ 
$P_{0}^{\phi }$ and $d\sigma $ by $\Delta _{b},$ $P_{0}$ and $d\mu ,$
respectively.

First we compare the relation between $P_{0}^{\phi }$ and $P_{0}.$

\begin{lemma}
Let $(M,J,\theta )$ be a compact strictly pseudoconvex CR $(2n+1)$-manifold.
We obtain%
\begin{equation}
\begin{array}{c}
P_{0}^{\phi }f=P_{0}f-4\langle Pf+\overline{P}f,d_{b}\phi \rangle -2\mathcal{%
L}\left\langle \nabla _{b}\phi ,\nabla _{b}f\right\rangle -2n\langle J\nabla
_{b}\phi ,\nabla _{b}f_{0}\rangle -2n^{2}f_{0}\phi _{0}.%
\end{array}
\label{4a}
\end{equation}
\end{lemma}

\begin{proof}
By the definition of $P_{0}^{\phi }$ and (\ref{40}), we compute%
\begin{equation*}
\begin{array}{lll}
P_{0}^{\phi }f & = & 4[\delta _{b}(P^{\phi }f)+\overline{\delta }_{b}(%
\overline{P}^{\phi }f)]-4\langle P^{\phi }f+\overline{P}^{\phi }f,d_{b}\phi
\rangle \\ 
& = & 4(P_{\beta }f-\frac{1}{2}\left\langle \nabla _{b}\phi ,\nabla
_{b}f\right\rangle _{,\beta }+\frac{n}{2}\sqrt{-1}f_{0}\phi _{\beta
})^{,\beta } \\ 
&  & +4(P_{\overline{\beta }}f-\frac{1}{2}\left\langle \nabla _{b}\phi
,\nabla _{b}f\right\rangle _{,\overline{\beta }}-\frac{n}{2}\sqrt{-1}%
f_{0}\phi _{\overline{\beta }})^{,\overline{\beta }} \\ 
&  & -4\langle Pf+\overline{P}f,d_{b}\phi \rangle -2\left\langle \nabla
_{b}\left\langle \nabla _{b}\phi ,\nabla _{b}f\right\rangle ,\nabla _{b}\phi
\right\rangle \\ 
& = & P_{0}f-4\langle Pf+\overline{P}f,d_{b}\phi \rangle -2\mathcal{L}%
\left\langle \nabla _{b}\phi ,\nabla _{b}f\right\rangle \\ 
&  & -2n\left( \langle J\nabla _{b}\phi ,\nabla _{b}f_{0}\rangle +nf_{0}\phi
_{0}\right) .%
\end{array}%
\end{equation*}
\end{proof}

Second, we define the weighted Kohn Laplacian operator as 
\begin{equation*}
\begin{array}{c}
\square _{b}^{\phi }f:=(-\mathcal{L}+n\sqrt{-1}T)f,%
\end{array}%
\end{equation*}%
then we have the similar formula for $P_{0}^{\phi }$ like the expression for 
$P_{0}$ (\ref{3}).

\begin{lemma}
\label{lemma}Let $(M,J,\theta )$ be a compact strictly pseudoconvex CR $%
(2n+1)$-manifold. We have%
\begin{equation}
\begin{array}{lll}
P_{0}^{\phi }f & = & 2\left( \mathcal{L}^{2}+n^{2}T^{2}\right) f-4n{\func{Re}%
}Q^{\phi }f-2n^{2}\phi _{0}f_{0} \\ 
& = & 2\square _{b}^{\phi }\overline{\square }_{b}^{\phi }f-4nQ^{\phi
}f-2n^{2}\phi _{0}f_{0}-2n\sqrt{-1}\left\langle \nabla _{b}\phi _{0},\nabla
_{b}f\right\rangle \\ 
& = & 2\overline{\square }_{b}^{\phi }\square _{b}^{\phi }f-4n\overline{Q}%
^{\phi }f-2n^{2}\phi _{0}f_{0}+2n\sqrt{-1}\left\langle \nabla _{b}\phi
_{0},\nabla _{b}f\right\rangle .%
\end{array}
\label{4}
\end{equation}
\end{lemma}

\begin{proof}
By the straightforward calculation, we have%
\begin{equation*}
\begin{array}{lll}
\square _{b}^{\phi }\overline{\square }_{b}^{\phi }f & = & (-\mathcal{L}+n%
\sqrt{-1}T)(-\mathcal{L}f-n\sqrt{-1}f_{0}) \\ 
& = & \left( \mathcal{L}^{2}f+n^{2}f_{00}\right) +n\sqrt{-1}[\mathcal{L},T]f
\\ 
& = & \left( \mathcal{L}^{2}f+n^{2}f_{00}\right) +n\sqrt{-1}\left( 2{\func{Im%
}}Q^{\phi }f+\left\langle \nabla _{b}\phi _{0},\nabla _{b}f\right\rangle
\right)%
\end{array}%
\end{equation*}%
and%
\begin{equation*}
\begin{array}{lll}
\square _{b}^{\phi }\overline{\square }_{b}^{\phi }f & = & (-\mathcal{L}+n%
\sqrt{-1}T)(-\mathcal{L}f-n\sqrt{-1}f_{0}) \\ 
& = & (-\mathcal{L}+n\sqrt{-1}T)(-\Delta _{b}f+\left\langle \nabla _{b}\phi
,\nabla _{b}f\right\rangle -n\sqrt{-1}f_{0}) \\ 
& = & \mathcal{L}(\Delta _{b}f-\left\langle \nabla _{b}\phi ,\nabla
_{b}f\right\rangle +n\sqrt{-1}f_{0}) \\ 
&  & -n\sqrt{-1}T(\Delta _{b}f-\left\langle \nabla _{b}\phi ,\nabla
_{b}f\right\rangle +n\sqrt{-1}f_{0}) \\ 
& = & \Delta _{b}(\Delta _{b}f-\left\langle \nabla _{b}\phi ,\nabla
_{b}f\right\rangle +n\sqrt{-1}f_{0}) \\ 
&  & -\left\langle \nabla _{b}\phi ,\nabla _{b}(\Delta _{b}f-\left\langle
\nabla _{b}\phi ,\nabla _{b}f\right\rangle +n\sqrt{-1}f_{0})\right\rangle \\ 
&  & -n\sqrt{-1}T(\Delta _{b}f-\left\langle \nabla _{b}\phi ,\nabla
_{b}f\right\rangle +n\sqrt{-1}f_{0}) \\ 
& = & \Delta _{b}^{2}f+n^{2}f_{00}-\mathcal{L}\left\langle \nabla _{b}\phi
,\nabla _{b}f\right\rangle +n\sqrt{-1}[\mathcal{L},T]f-\left\langle \nabla
_{b}\phi ,\nabla _{b}\Delta _{b}f\right\rangle \\ 
& = & \Delta _{b}^{2}f+n^{2}f_{00}-\mathcal{L}\left\langle \nabla _{b}\phi
,\nabla _{b}f\right\rangle +n\sqrt{-1}\mathcal{L}f_{0}-\left\langle \nabla
_{b}\phi ,\nabla _{b}\Delta _{b}f\right\rangle \\ 
&  & -n\sqrt{-1}[\Delta _{b}f_{0}-2{\func{Im}}Qf-f_{\overline{\beta }}\phi
_{\beta 0}-f_{\beta }\phi _{\overline{\beta }0} \\ 
&  & \text{ \ \ \ \ \ \ \ \ \ \ \ }-\left\langle \nabla _{b}\phi ,\nabla
_{b}f_{0}\right\rangle +A^{\overline{\alpha }\overline{\beta }}f_{\overline{%
\alpha }}\phi _{\overline{\beta }}+A^{\alpha \beta }f_{\alpha }\phi _{\beta }%
\mathcal{]} \\ 
& = & \Delta _{b}^{2}f+n^{2}f_{00}-\mathcal{L}\left\langle \nabla _{b}\phi
,\nabla _{b}f\right\rangle -\left\langle \nabla _{b}\phi ,\nabla _{b}\Delta
_{b}f\right\rangle \\ 
&  & +n\sqrt{-1}[2{\func{Im}}Qf+(f_{\overline{\beta }}\phi _{0\beta
}+f_{\beta }\phi _{0\overline{\beta }})-2(A^{\overline{\alpha }\overline{%
\beta }}f_{\overline{\alpha }}\phi _{\overline{\beta }}+A^{\alpha \beta
}f_{\alpha }\phi _{\beta })]%
\end{array}%
\end{equation*}%
By using the equation%
\begin{equation*}
\begin{array}{c}
n\langle J\nabla _{b}\phi ,\nabla _{b}f_{0}\rangle =\langle \nabla _{b}\phi
,\nabla _{b}\Delta _{b}f\rangle -2\langle Pf+\overline{P}f,d_{b}\phi \rangle
+2n\sqrt{-1}(A^{\overline{\alpha }\overline{\beta }}f_{\overline{\alpha }%
}\phi _{\overline{\beta }}-A^{\alpha \beta }f_{\alpha }\phi _{\beta }),%
\end{array}%
\end{equation*}%
we deduce%
\begin{equation*}
\begin{array}{lll}
\square _{b}^{\phi }\overline{\square }_{b}^{\phi }f & = & \Delta
_{b}^{2}f+n^{2}f_{00}-\mathcal{L}\left\langle \nabla _{b}\phi ,\nabla
_{b}f\right\rangle -n\left\langle J\nabla _{b}\phi ,\nabla
_{b}f_{0}\right\rangle -2\langle Pf+\overline{P}f,d_{b}\phi \rangle \\ 
&  & +2n\sqrt{-1}(A^{\overline{\alpha }\overline{\beta }}f_{\overline{\alpha 
}}\phi _{\overline{\beta }}-A^{\alpha \beta }f_{\alpha }\phi _{\beta })+2n{%
\func{Im}}Qf \\ 
&  & +n\sqrt{-1}\left\langle \nabla _{b}\phi _{0},\nabla _{b}f\right\rangle
-2n\sqrt{-1}(A^{\overline{\alpha }\overline{\beta }}f_{\overline{\alpha }%
}\phi _{\overline{\beta }}+A^{\alpha \beta }f_{\alpha }\phi _{\beta }) \\ 
& = & \frac{1}{2}P_{0}f+2n{\func{Re}}Qf-\mathcal{L}\left\langle \nabla
_{b}\phi ,\nabla _{b}f\right\rangle -n\left\langle J\nabla _{b}\phi ,\nabla
_{b}f_{0}\right\rangle -2\langle Pf+\overline{P}f,d_{b}\phi \rangle \\ 
&  & +2n\func{Im}Qf-4n\sqrt{-1}A^{\alpha \beta }f_{\alpha }\phi _{\beta }+n%
\sqrt{-1}\left\langle \nabla _{b}\phi _{0},\nabla _{b}f\right\rangle .%
\end{array}%
\end{equation*}%
Finally, we obtain 
\begin{equation*}
\begin{array}{lll}
2\square _{b}^{\phi }\overline{\square }_{b}^{\phi }f & = & P_{0}f-2\mathcal{%
L}\left\langle \nabla _{b}\phi ,\nabla _{b}f\right\rangle -2n\left\langle
J\nabla _{b}\phi ,\nabla _{b}f_{0}\right\rangle -4\langle Pf+\overline{P}%
f,d_{b}\phi \rangle \\ 
&  & +4n(Qf-2\sqrt{-1}A^{\alpha \beta }f_{\alpha }\phi _{\beta })+2n\sqrt{-1}%
\left\langle \nabla _{b}\phi _{0},\nabla _{b}f\right\rangle \\ 
& = & P_{0}^{\phi }f+2n^{2}\phi _{0}f_{0}+nQ^{\phi }f+2n\sqrt{-1}%
\left\langle \nabla _{b}\phi _{0},\nabla _{b}f\right\rangle ,%
\end{array}%
\end{equation*}%
as desired.
\end{proof}

In the following we show that if the pseudohermitian torsion of $M$ is zero
and $\phi _{0}$ vanishes, then the weighted CR Paneitz operator $P_{0}^{\phi
}$ is nonnegative for all smooth functions with Dirichlet boundary condition
(\ref{1}).

\begin{lemma}
\label{l41} Let $(M,J,\theta ,d\sigma )$ be a compact weighted Sasakian $%
(2n+1)$-manifold with boundary $\Sigma $. If $\phi _{0}$ vanishes, then the
weighted CR Paneitz operator $P_{0}^{\phi }$ is nonnegative for all smooth
functions with Dirichlet boundary condition (\ref{1}). In particular, the
weighted CR Paneitz operator $P_{0}^{\phi }$ is nonnegative in a closed
weighted Sasakian $(2n+1)$-manifold if $\phi _{0}$ vanishes.
\end{lemma}

\begin{proof}
The zero pseudohermitian torsion and $\phi _{0}=0$ implies that the weighted
CR Paneitz operator becomes $P_{0}^{\phi }=\square _{b}^{\phi }\overline{%
\square }_{b}^{\phi }+\overline{\square }_{b}^{\phi }\square _{b}^{\phi },$
and the weighted Kohn Laplacian $\square _{b}^{\phi }$ and $\overline{%
\square }_{b}^{\phi }$ commute, so they are diagonalized simultaneously on
the finite dimensional eigenspace of $\square _{b}^{\phi }$ with respect to
any nonzero eigenvalue. And we know that the eigenvalues of $\square
_{b}^{\phi }$ (and thus of $\overline{\square }_{b}^{\phi }$) are all
nonnegative, since for any real function $\varphi $ with $\varphi =0$ on $%
\Sigma $ 
\begin{equation*}
\begin{array}{lll}
\int_{M}\varphi \square _{b}^{\phi }\varphi d\sigma & = & \int_{M}\varphi (-%
\mathcal{L}\varphi +n\sqrt{-1}\varphi _{0})d\sigma \\ 
& = & \int_{M}[\left\vert \nabla _{b}\varphi \right\vert ^{2}+\frac{1}{2}n%
\sqrt{-1}\varphi ^{2}\phi _{0}]d\sigma -\frac{1}{2}C_{n}\int_{\Sigma
}\varphi \lbrack \varphi _{e_{_{2n}}}+n\sqrt{-1}\alpha \varphi ]d\Sigma
_{p}^{\phi } \\ 
& = & \int_{M}\left\vert \nabla _{b}\varphi \right\vert ^{2}d\sigma ,%
\end{array}%
\end{equation*}%
here we used the condition $\phi _{0}$ vanishes on $M$. Therefore, $%
P_{0}^{\phi }$ is nonnegative.
\end{proof}

\begin{lemma}
Let $(M,J,\theta )$ be a compact Sasakian $3$-manifold with the smooth
boundary $\Sigma $. Then for the first eigenfunction $f$ of Dirichlet
eigenvalue problem 
\begin{equation*}
\left\{ 
\begin{array}{ccll}
\Delta _{b}f & = & -\lambda _{1}f & \mathrm{on\ }M, \\ 
f & = & 0 & \mathrm{on\ }\Sigma ,%
\end{array}%
\right.
\end{equation*}%
we have%
\begin{equation*}
\begin{array}{c}
\int_{M}fP_{0}fd\mu =\frac{16}{\lambda _{1}}\int_{M}\left\vert
P_{1}f\right\vert ^{2}d\mu \geq 0.%
\end{array}%
\end{equation*}
\end{lemma}

\begin{proof}
It follows from the formula $\left[ \Delta _{b},T\right] f=4{\func{Im}}[%
\sqrt{-1}\left( A_{\overline{1}\overline{1}}f_{1}\right) ,_{1}]=0$ that $%
\Delta _{b}f_{0}=-\lambda _{1}f_{0}.$ From the divergence theorem, we
compute 
\begin{equation}
\begin{array}{l}
\int_{M}[\left\vert \nabla _{b}f\right\vert ^{2}-\lambda _{1}f^{2}]d\mu
=\int_{M}\left\vert \nabla _{b}f\right\vert ^{2}+f\Delta _{b}fd\mu
=\int_{\Sigma }ff_{e_{2}}\theta \wedge e^{1}=0%
\end{array}
\label{27}
\end{equation}%
and 
\begin{equation}
\begin{array}{c}
\int_{M}[\left\vert \nabla _{b}f_{0}\right\vert ^{2}-\lambda
_{1}f_{0}^{2}]d\mu =\int_{M}\left\vert \nabla _{b}f_{0}\right\vert
^{2}+f_{0}\Delta _{b}f_{0}d\mu =\int_{\Sigma }f_{0}f_{0e_{2}}\theta \wedge
e^{1}=0,%
\end{array}
\label{28}
\end{equation}%
here we used the last equation is zero which will be showed later. From the
identity $P_{0}f=2\left( \Delta _{b}^{2}f+f_{00}\right) =2\left( \lambda
_{1}^{2}f+f_{00}\right) $ on $M$ and from (\ref{c}), we have 
\begin{equation}
\begin{array}{c}
\int_{M}fP_{0}fd\mu =2\int_{M}\left( \lambda _{1}^{2}f^{2}-f_{0}^{2}\right)
d\mu -4\int_{\Sigma }\alpha ff_{0}\theta \wedge e^{1}=2\int_{M}\left(
\lambda _{1}^{2}f^{2}-f_{0}^{2}\right) d\mu .%
\end{array}
\label{29}
\end{equation}%
Also from the identity $P_{1}f=f_{\overline{1}11}=\frac{1}{2}\left( \Delta
_{b}f-\sqrt{-1}f_{0}\right) ,_{1}=-\frac{1}{2}\left( \lambda _{1}f_{1}+\sqrt{%
-1}f_{01}\right) ,$ we get 
\begin{equation*}
\begin{array}{c}
\left\vert P_{1}f\right\vert ^{2}=\frac{1}{8}[\lambda _{1}^{2}\left\vert
\nabla _{b}f\right\vert ^{2}+\left\vert \nabla _{b}f_{0}\right\vert
^{2}+2\lambda _{1}\langle J\nabla _{b}f,\nabla _{b}f_{0}\rangle ].%
\end{array}%
\end{equation*}%
Thus, from (\ref{27}), (\ref{28}), (\ref{d}) and (\ref{29}), we have%
\begin{equation*}
\begin{array}{lll}
8\int_{M}\left\vert P_{1}f\right\vert ^{2}d\mu & = & \int_{M}[\lambda
_{1}^{2}\left\vert \nabla _{b}f\right\vert ^{2}+\left\vert \nabla
_{b}f_{0}\right\vert ^{2}-2\lambda _{1}f_{0}^{2}]d\mu \\ 
& = & \int_{M}[\lambda _{1}^{3}f^{2}-\lambda _{1}f_{0}^{2}]d\mu =\frac{1}{2}%
\lambda _{1}\int_{M}fP_{0}fd\mu ,%
\end{array}%
\end{equation*}%
as desired.

In the following we claim that $\int_{\Sigma }f_{0}f_{0e_{2}}\theta \wedge
e^{1}=0.$ By using the equation 
\begin{equation}
\begin{array}{c}
\int_{M}\psi _{0}d\mu =\int_{M}\mathrm{div}_{b}\left( J\nabla _{b}\psi
\right) d\mu =\int_{\Sigma }\psi _{e_{1}}\theta \wedge e^{1}=0%
\end{array}
\label{30}
\end{equation}%
for any real function $\psi $ which vanishes on $\Sigma $ to get that $%
\int_{M}\left( f^{2}\right) _{0}d\mu =0,$ and from%
\begin{equation*}
\begin{array}{c}
\int_{M}[f\Delta _{b}f_{0}+\left\langle \nabla _{b}f,\nabla
_{b}f_{0}\right\rangle ]d\mu =\int_{\Sigma }ff_{0e_{2}}\theta \wedge e^{1}=0,%
\end{array}%
\end{equation*}%
one obtains $\int_{M}\left\langle \nabla _{b}f,\nabla _{b}f_{0}\right\rangle
d\mu =0$. Thus%
\begin{equation}
\begin{array}{c}
0=\int_{M}[f_{0}\Delta _{b}f+\left\langle \nabla _{b}f,\nabla
_{b}f_{0}\right\rangle ]d\mu =\int_{\Sigma }f_{0}f_{e_{2}}\theta \wedge
e^{1}.%
\end{array}
\label{32}
\end{equation}%
On the other hand, since $f=0$ on $\Sigma $ and $\alpha e_{2}+T$ is tangent
along $\Sigma $ from the definition of $\alpha ,$ so on $\Sigma $ we have%
\begin{equation*}
\begin{array}{c}
f_{0}=-\alpha f_{e_{2}}%
\end{array}%
\end{equation*}%
and 
\begin{equation*}
\begin{array}{c}
-\alpha \theta \wedge e^{1}=e^{1}\wedge e^{2}=\frac{1}{2}d\theta%
\end{array}%
\end{equation*}%
which is a nonnegative $2$-form on $\Sigma $. It follows from (\ref{32})
that 
\begin{equation*}
\begin{array}{c}
0=\int_{\Sigma }f_{0}f_{e_{2}}\theta \wedge e^{1}=-\int_{\Sigma }\alpha
f_{e_{2}}^{2}\theta \wedge e^{1}=\int_{\Sigma }f_{e_{2}}^{2}e^{1}\wedge e^{2}%
\end{array}%
\end{equation*}%
to imply that 
\begin{equation*}
\begin{array}{l}
\int_{\Sigma }f_{0}f_{0e_{2}}\theta \wedge e^{1}=-\int_{\Sigma }\alpha
f_{e_{2}}f_{0e_{2}}\theta \wedge e^{1}=\int_{\Sigma
}f_{e_{2}}f_{0e_{2}}e^{1}\wedge e^{2}=0,%
\end{array}%
\end{equation*}%
here we used the H\H{o}lder's inequality in the last equation.
\end{proof}

Finally, we recall the Lemma\ 4.1 in \cite{ccw}.

\begin{lemma}
\label{lemma 3.1}Let $(M,J,\theta )$ be a compact strictly pseudoconvex CR $%
(2n+1)$-manifold with the smooth boundary $\Sigma $ if $n\geq 2$. Then for
the first eigenfunction $f$ of Dirichlet eigenvalue problem 
\begin{equation*}
\left\{ 
\begin{array}{ccll}
\Delta _{b}f & = & -\lambda _{1}f & \mathrm{on\ }M, \\ 
f & = & 0 & \mathrm{on\ }\Sigma ,%
\end{array}%
\right.
\end{equation*}%
we have%
\begin{equation*}
\begin{array}{l}
\frac{n-1}{8n}\int_{M}fP_{0}fd\mu =\int_{M}\sum_{\beta ,\gamma }|f_{\beta 
\overline{\gamma }}-\frac{1}{n}f_{\sigma }{}^{\sigma }h_{\beta \overline{%
\gamma }}|^{2}d\mu +\frac{1}{16}C_{n}\int_{\Sigma }(H_{p.h}+\tilde{\omega}%
_{n}^{\;n}(e_{n}))f_{e_{2n}}^{2}d\Sigma _{p}%
\end{array}%
\end{equation*}%
which implies 
\begin{equation*}
\begin{array}{c}
\int_{M}fP_{0}fd\mu \geq 0%
\end{array}%
\end{equation*}%
if $H_{p.h}+\tilde{\omega}_{n}^{\;n}(e_{n})$ is nonnegative on $\Sigma $.
\end{lemma}

\section{The Weighted CR Reilly Formula}

Let $M$ be a compact strictly pseudoconvex CR $(2n+1)$-manifold with
boundary $\Sigma $. We write $\theta _{\gamma }^{\;\text{\ }\beta }=\omega
_{\gamma }^{\;\text{\ }\beta }+\sqrt{-1}\tilde{\omega}_{\gamma }^{\;\text{\ }%
\beta }$ with $\omega _{\gamma }^{\;\text{\ }\beta }=\mathrm{Re}(\theta
_{\gamma }^{\;\text{\ }\beta })$, $\tilde{\omega}_{\gamma }^{\;\text{\ }%
\beta }=\mathrm{Im}(\theta _{\gamma }^{\;\text{\ }\beta })$ and $Z_{\beta }=%
\frac{1}{2}(e_{\beta }-\sqrt{-1}e_{n+\beta })$ for real vectors $e_{\beta }$%
, $e_{n+\beta }$, $\beta =1,\cdots ,n$. It follows that $e_{n+\beta
}=Je_{\beta }$. Let $e^{\beta }=\mathrm{Re}(\theta ^{\beta })$, $e^{n+\beta
}=\mathrm{Im}(\theta ^{\beta })$, $\beta =1,\cdots ,n$. Then $\{\theta
,e^{\beta },e^{n+\beta }\}$ is dual to $\{T,e_{\beta },e_{n+\beta }\}$. Now
in view of (\ref{dtheta}) and (\ref{structure equs}), we have the following
real version of structure equations: 
\begin{equation*}
\left\{ 
\begin{array}{l}
d\theta =2\sum_{\beta }e^{\beta }\wedge e^{n+\beta }, \\ 
\nabla e_{\gamma }=\omega _{\gamma }^{\;\text{\ }\beta }\otimes e_{\beta }+%
\tilde{\omega}_{\gamma }^{\;\text{\ }\beta }\otimes e_{n+\beta },\text{ }%
\nabla e_{n+\gamma }=\omega _{\gamma }^{\;\text{\ }\beta }\otimes e_{n+\beta
}-\tilde{\omega}_{\gamma }^{\;\text{\ }\beta }\otimes e_{\beta }, \\ 
de^{\gamma }=e^{\beta }\wedge \omega _{\beta }^{\;\text{\ }\gamma
}-e^{n+\beta }\wedge \tilde{\omega}_{\beta }^{\;\text{\ }\gamma }\text{ 
\textrm{mod} }\theta ;\text{ }de^{n+\gamma }=e^{\beta }\wedge \tilde{\omega}%
_{\beta }^{\text{ \ }\gamma }+e^{n+\beta }\wedge \omega _{\beta }^{\text{ \ }%
\gamma }\text{ \textrm{mod} }\theta .%
\end{array}%
\right.
\end{equation*}

Let $\Sigma $ be a surface contained in $M$. The singular set $S_{\Sigma }$
consists of those points where $\xi $ coincides with the tangent bundle $%
T\Sigma $ of $\Sigma $. It is easy to see that $S_{\Sigma }$ is a closed
set. On $\xi ,$ we can associate a natural metric $\langle $ $,$ $\rangle =%
\frac{1}{2}d\theta (\cdot ,J\cdot )$ call the Levi metric. For a vector $%
v\in \xi ,$ we define the length of $v$ by $\left\vert v\right\vert
^{2}=\langle v,v\rangle .$ With respect to the Levi metric, we can take unit
vector fields $e_{1},\cdots ,e_{2n-1}\in \xi \cap T\Sigma $ on $\Sigma
\backslash S_{\Sigma }$, called the characteristic fields and $e_{2n}=Je_{n}$%
, called the Legendrian normal. The $p$(pseudohermitian)-mean curvature $%
H_{p.h}$ on $\Sigma \backslash S_{\Sigma }$ is defined by 
\begin{equation*}
\begin{array}{c}
H_{p.h}=\sum_{j=1}^{2n-1}\left\langle \nabla
_{e_{j}}e_{2n},e_{j}\right\rangle =-\sum_{j=1}^{2n-1}\left\langle \nabla
_{e_{j}}e_{j},e_{2n}\right\rangle .%
\end{array}%
\end{equation*}%
For $e_{1},\cdots ,e_{2n-1}$ being characteristic fields, we have the $p$%
-area element 
\begin{equation*}
\begin{array}{c}
d\Sigma _{p}=\theta \wedge e^{1}\wedge e^{n+1}\wedge \cdots \wedge
e^{n-1}\wedge e^{2n-1}\wedge e^{n}%
\end{array}%
\end{equation*}%
on $\Sigma $ and all surface integrals over $\Sigma $ are with respect to
this $2n$-form $d\Sigma _{p}$. Note that $d\Sigma _{p}$ continuously extends
over the singular set $S_{\Sigma }$ and vanishes on $S_{\Sigma }$.

We also write $f_{e_{j}}=e_{j}f$ and $\nabla _{b}f=\frac{1}{2}(f_{e_{\beta
}}e_{\beta }+f_{e_{n+\beta }}e_{n+\beta })$. Moreover, $%
f_{e_{j}e_{k}}=e_{k}e_{j}f-\nabla _{e_{k}}e_{j}f$ and $\Delta _{b}f=\frac{1}{%
2}\sum_{\beta }(f_{e_{\beta }e_{\beta }}+f_{e_{n+\beta }e_{n+\beta }})$.
Next we define the subdivergence operator $div_{b}(\cdot )$ by $%
div_{b}(W)=W^{\beta },_{\beta }+W^{\overline{\beta }},_{\overline{\beta }}$
for all vector fields $W=W^{\beta }Z_{\beta }+W^{\overline{\beta }}Z_{%
\overline{\beta }}$ and its real version is $div_{b}(W)=\varphi _{\beta
,e_{\beta }}+\psi _{n+\beta ,e_{n+\beta }}$ for $W=\varphi _{\beta }e_{\beta
}+\psi _{n+\beta }e_{n+\beta }$. We define the tangential subgradient $%
\nabla _{b}^{t}$ of a function $f$ by $\nabla _{b}^{t}f=\nabla _{b}f-\langle
\nabla _{b}f,e_{2n}\rangle e_{2n}$ and the tangent sub-Laplacian $\Delta
_{b}^{t}$ of $f$ by $\Delta _{b}^{t}f=\frac{1}{2}%
\sum_{j=1}^{2n-1}[(e_{j})^{2}-(\nabla _{e_{j}}e_{j})^{t}]f,$ where $(\nabla
_{e_{j}}e_{j})^{t}$ is the tangential part of $\nabla _{e_{j}}e_{j}$.

We first recall the following CR Bochner formula for $\Delta _{b}$.

\begin{lemma}
Let $(M,J,\theta )$ be a strictly pseudoconvex CR $(2n+1)$-manifold. For a
real function $f$, we have%
\begin{equation}
\begin{array}{lll}
\frac{1}{2}\Delta _{b}|\nabla _{b}f|^{2} & = & |(\nabla
^{H})^{2}f|^{2}+\langle \nabla _{b}f,\nabla _{b}\Delta _{b}f\rangle \\ 
&  & +[2Ric-(n-2)Tor](\nabla _{b}f_{\mathbb{C}},\nabla _{b}f_{\mathbb{C}})
\\ 
&  & +2\langle J\nabla _{b}f,\nabla _{b}f_{0}\rangle ,%
\end{array}
\label{A}
\end{equation}%
where $\nabla _{b}f_{\mathbb{C}}=f^{\beta }Z_{\beta }$ is the corresponding
complex $(1,0)$-vector field of $\nabla _{b}f$.
\end{lemma}

Now we derive the following CR Bochner formula for $\mathcal{L}$.

\begin{lemma}
Let $(M,J,\theta )$ be a strictly pseudoconvex CR $(2n+1)$-manifold. For a
real function $f$, we have%
\begin{equation}
\begin{array}{lll}
\frac{1}{2}\mathcal{L}|\nabla _{b}f|^{2} & = & |(\nabla
^{H})^{2}f|^{2}+\langle \nabla _{b}f,\nabla _{b}\mathcal{L}f\rangle \\ 
&  & +[Ric+(\nabla ^{H})^{2}\phi -\frac{n-2}{2}Tor](\nabla _{b}f,\nabla
_{b}f) \\ 
&  & +2\langle J\nabla _{b}f,\nabla _{b}f_{0}\rangle -f_{0}\langle J\nabla
_{b}f,\nabla _{b}\phi \rangle ,%
\end{array}
\label{Bochnerformula}
\end{equation}%
where $(\nabla ^{H})^{2}\phi (\nabla _{b}f,\nabla _{b}f)=\phi _{\beta \gamma
}f^{\beta }f^{\gamma }+\phi _{\overline{\beta }\overline{\gamma }}f^{%
\overline{\beta }}f^{\overline{\gamma }}+\phi _{\beta \overline{\gamma }%
}f^{\beta }f^{\overline{\gamma }}+\phi _{\overline{\beta }\gamma }f^{%
\overline{\beta }}f^{\gamma }.$
\end{lemma}

The proof of the above formula follows from the definition of $\mathcal{L}$
and the identity 
\begin{equation}
\begin{array}{c}
\langle \nabla _{b}f,\nabla _{b}\langle \nabla _{b}f,\nabla _{b}\phi \rangle
\rangle -\frac{1}{2}\langle \nabla _{b}\phi ,\nabla _{b}|\nabla
_{b}f|^{2}\rangle =(\nabla ^{H})^{2}\phi (\nabla _{b}f,\nabla
_{b}f)-f_{0}\langle J\nabla _{b}f,\nabla _{b}\phi \rangle .%
\end{array}
\label{10}
\end{equation}%
Also we note that%
\begin{equation}
\begin{array}{ccc}
\langle J\nabla _{b}f,\nabla _{b}f_{0}\rangle & = & \frac{1}{n}\langle
\nabla _{b}f,\nabla _{b}\mathcal{L}f\rangle -\frac{2}{n}\langle P^{\phi }f+%
\overline{P}^{\phi }f,d_{b}f\rangle \\ 
&  & +f_{0}\langle J\nabla _{b}f,\nabla _{b}\phi \rangle -Tor(\nabla
_{b}f,\nabla _{b}f).%
\end{array}
\label{10a}
\end{equation}%
Then the CR Bochner formula for $\mathcal{L}$ becomes%
\begin{equation}
\begin{array}{lll}
\frac{1}{2}\mathcal{L}|\nabla _{b}f|^{2} & = & |(\nabla ^{H})^{2}f|^{2}+%
\frac{n+2}{n}\langle \nabla _{b}f,\nabla _{b}\mathcal{L}f\rangle \\ 
&  & +[Ric+(\nabla ^{H})^{2}\phi -\frac{n+2}{2}Tor](\nabla _{b}f,\nabla
_{b}f) \\ 
&  & -\frac{4}{n}\langle P^{\phi }f+\overline{P}^{\phi }f,d_{b}f\rangle
+f_{0}\langle J\nabla _{b}f,\nabla _{b}\phi \rangle .%
\end{array}
\label{11}
\end{equation}

For the proof of the weighted CR Reilly formula, we need a series of
formulae as we derived in \cite{ccw}.

\begin{lemma}
Let $(M,J,\theta ,d\sigma )$ be a compact weighted strictly pseudoconvex CR $%
(2n+1)$-manifold with boundary $\Sigma $. For real functions $f$ and $g$, we
have%
\begin{equation}
\begin{array}{c}
\int_{M}\left( \mathcal{L}f\right) d\sigma =\frac{1}{2}C_{n}\int_{\Sigma
}f_{e_{_{2n}}}d\Sigma _{p}^{\phi },%
\end{array}
\label{a}
\end{equation}%
\begin{equation}
\begin{array}{c}
\int_{M}[g\mathcal{L}f+\langle \nabla _{b}f,\nabla _{b}g\rangle ]d\sigma =%
\frac{1}{2}C_{n}\int_{\Sigma }gf_{e_{_{2n}}}d\Sigma _{p}^{\phi },%
\end{array}
\label{b}
\end{equation}%
\begin{equation}
\begin{array}{c}
\int_{M}[ff_{00}-ff_{0}\phi _{0}+f_{0}^{2}]d\sigma =-C_{n}\int_{\Sigma
}\alpha ff_{0}d\Sigma _{p}^{\phi },%
\end{array}
\label{c}
\end{equation}%
\begin{equation}
\begin{array}{c}
\int_{M}[\langle J\nabla _{b}f,\nabla _{b}f_{0}\rangle -f_{0}\langle J\nabla
_{b}f,\nabla _{b}\phi \rangle +nf_{0}^{2}]d\sigma =\frac{1}{2}%
C_{n}\int_{\Sigma }f_{0}f_{e_{n}}d\Sigma _{p}^{\phi },%
\end{array}
\label{d}
\end{equation}%
\begin{equation}
\begin{array}{c}
\int_{M}[\langle P^{\phi }f+\overline{P}^{\phi }f,d_{b}f\rangle +\frac{1}{4}%
fP_{0}^{\phi }f]d\sigma =\frac{1}{2}\sqrt{-1}C_{n}\int_{\Sigma
}f(P_{n}^{\phi }f-P_{\overline{n}}^{\phi }f)d\Sigma _{p}^{\phi }.%
\end{array}
\label{e}
\end{equation}%
Here$\ d\sigma =e^{-\phi }\theta \wedge (d\theta )^{n}$ is the weighted
volume measure and $d\Sigma _{p}^{\phi }=e^{-\phi }d\Sigma _{p}$ is the
weighted $p$-area element of $\Sigma ,$ and $C_{n}=2^{n}n!$.
\end{lemma}

\begin{lemma}
Let $(M,J,\theta ,d\sigma )$ be a compact weighted pseudohermitian $(2n+1)$%
-manifold with boundary $\Sigma $. For any real-valued function $f$ on $%
\Sigma ,$ we have%
\begin{equation}
\begin{array}{c}
\int_{\Sigma }[f_{e_{n}}+(2\alpha -\phi _{e_{n}})f]d\Sigma _{p}^{\phi }=0,%
\end{array}
\label{f}
\end{equation}%
\begin{equation}
\begin{array}{c}
\int_{\Sigma }[f_{\overline{\beta }}+(\sum_{\gamma \neq n}\theta _{\overline{%
\beta }}^{\;\text{\ }\overline{\gamma }}(Z_{\overline{\gamma }})+\frac{1}{2}%
\theta _{\overline{\beta }}^{\;\text{\ }\overline{n}}(e_{n})-\phi _{%
\overline{\beta }})f]d\Sigma _{p}^{\phi }=0\text{ \textrm{for any} }\beta
\neq n.%
\end{array}
\label{g}
\end{equation}
\end{lemma}

\textbf{The Proof of Theorem}\textup{\textbf{\ \ref{Reilly'sformula}:}}

\begin{proof}
By integrating the CR Bochner formula (\ref{11}) for $\mathcal{L}$, from (%
\ref{b}) and (\ref{e}), we have%
\begin{equation*}
\begin{array}{lll}
\frac{1}{2}\int_{M}\mathcal{L}|\nabla _{b}f|^{2}d\sigma & = & 
\int_{M}|(\nabla ^{H})^{2}f|^{2}d\sigma +\int_{M}[Ric+(\nabla ^{H})^{2}\phi -%
\frac{n+2}{2}Tor](\nabla _{b}f,\nabla _{b}f)d\sigma \\ 
&  & -\frac{n+2}{n}\int_{M}(\mathcal{L}f)^{2}d\sigma +\frac{1}{n}%
\int_{M}fP_{0}^{\phi }fd\sigma +\int_{M}f_{0}\langle J\nabla _{b}f,\nabla
_{b}\phi \rangle d\sigma \\ 
&  & +\frac{n+2}{2n}C_{n}\int_{\Sigma }(\mathcal{L}f)f_{e_{_{2n}}}d\Sigma
_{p}^{\phi }-\frac{2}{n}\sqrt{-1}C_{n}\int_{\Sigma }f(P_{n}^{\phi }f-P_{%
\overline{n}}^{\phi }f)d\Sigma _{p}^{\phi }.%
\end{array}%
\end{equation*}%
By combining (\ref{10a}), (\ref{d}) and (\ref{e}), we have%
\begin{equation}
\begin{array}{lll}
n^{2}\int_{M}f_{0}^{2}d\sigma & = & \int_{M}\left( \mathcal{L}f\right)
^{2}d\sigma -\frac{1}{2}\int_{M}fP_{0}^{\phi }fd\sigma +n\int_{M}Tor(\nabla
_{b}f,\nabla _{b}f)d\sigma \\ 
&  & +\frac{1}{2}C_{n}\int_{\Sigma }[nf_{0}f_{e_{n}}-(\mathcal{L}%
f)f_{e_{_{2n}}}]d\Sigma _{p}^{\phi }+\sqrt{-1}C_{n}\int_{\Sigma
}f(P_{n}^{\phi }f-P_{\overline{n}}^{\phi }f)d\Sigma _{p}^{\phi }.%
\end{array}
\label{7}
\end{equation}%
Also by (\ref{10}) and (\ref{b}), one gets 
\begin{equation}
\begin{array}{ll}
& \int_{M}f_{0}\langle J\nabla _{b}f,\nabla _{b}\phi \rangle d\sigma \\ 
= & \int_{M}(\nabla ^{H})^{2}\phi (\nabla _{b}f,\nabla _{b}f)d\sigma
+\int_{M}\left[ \left( \mathcal{L}f\right) \langle \nabla _{b}f,\nabla
_{b}\phi \rangle -\frac{1}{2}\left( \mathcal{L}\phi \right) |\nabla
_{b}f|^{2}\right] d\sigma \\ 
& +\frac{1}{4}C_{n}\int_{\Sigma }[|\nabla _{b}f|^{2}\phi
_{e_{_{2n}}}-2\langle \nabla _{b}f,\nabla _{b}\phi \rangle
f_{e_{_{2n}}}]d\Sigma _{p}^{\phi }.%
\end{array}
\label{8}
\end{equation}%
Note that $|(\nabla ^{H})^{2}f|^{2}=2\sum_{\beta ,\gamma }[|f_{\beta \gamma
}|^{2}+|f_{\beta \overline{\gamma }}|^{2}]$ and 
\begin{equation*}
\begin{array}{c}
\sum_{\beta ,\gamma }|f_{\beta \overline{\gamma }}|^{2}=\sum_{\beta ,\gamma
}|f_{\beta \overline{\gamma }}-\frac{1}{n}f_{\sigma }{}^{\sigma }h_{\beta 
\overline{\gamma }}|^{2}+\frac{1}{4n}\left( \Delta _{b}f\right) ^{2}+\frac{n%
}{4}f_{0}^{2}%
\end{array}%
\end{equation*}%
with $\Delta _{b}f=\mathcal{L}f+\langle \nabla _{b}f,\nabla _{b}\phi \rangle
.$ It follows from (\ref{7}) and (\ref{8})\ that%
\begin{equation}
\begin{array}{ll}
& \frac{1}{2}\int_{M}\mathcal{L}|\nabla _{b}f|^{2}d\sigma \\ 
= & 2\int_{M}\sum_{\beta ,\gamma }\left[ |f_{\beta \gamma }|^{2}+|f_{\beta 
\overline{\gamma }}-\frac{1}{n}f_{\sigma }{}^{\sigma }h_{\beta \overline{%
\gamma }}|^{2}\right] d\sigma -\frac{n+1}{n}\int_{M}(\mathcal{L}f)^{2}d\sigma
\\ 
& +\int_{M}[Ric+2(\nabla ^{H})^{2}\phi -\frac{n+1}{2}Tor](\nabla
_{b}f,\nabla _{b}f)d\sigma +\frac{3}{4n}\int_{M}fP_{0}^{\phi }fd\sigma \\ 
& +\frac{1}{2n}\int_{M}[2(n+1)\left( \mathcal{L}f\right) \langle \nabla
_{b}f,\nabla _{b}\phi \rangle -n\left( \mathcal{L}\phi \right) |\nabla
_{b}f|^{2}+\langle \nabla _{b}f,\nabla _{b}\phi \rangle ^{2}]d\sigma \\ 
& +\frac{2n+3}{4n}C_{n}\int_{\Sigma }(\mathcal{L}f)f_{e_{_{2n}}}d\Sigma
_{p}^{\phi }-\frac{3}{2n}\sqrt{-1}C_{n}\int_{\Sigma }f(P_{n}^{\phi }f-P_{%
\overline{n}}^{\phi }f)d\Sigma _{p}^{\phi } \\ 
& +\frac{1}{4}C_{n}\int_{\Sigma }[f_{0}f_{e_{n}}+|\nabla _{b}f|^{2}\phi
_{e_{_{2n}}}-2\langle \nabla _{b}f,\nabla _{b}\phi \rangle
f_{e_{_{2n}}}]d\Sigma _{p}^{\phi }.%
\end{array}
\label{9}
\end{equation}

In the following, we deal with the term $\int_{M}\sum_{\beta ,\gamma
}|f_{\beta \overline{\gamma }}-\frac{1}{n}f_{\sigma }{}^{\sigma }h_{\beta 
\overline{\gamma }}|^{2}d\sigma .$ The divergence formula for the trace-free
part of $f_{\beta \overline{\gamma }}$: 
\begin{equation*}
\begin{array}{c}
B_{\beta \overline{\gamma }}f=f_{\beta \overline{\gamma }}-\frac{1}{n}%
f_{\sigma }{}^{\sigma }h_{\beta \overline{\gamma }},%
\end{array}%
\end{equation*}%
is given by 
\begin{equation*}
\begin{array}{lll}
(B^{\beta \overline{\gamma }}f)(B_{\beta \overline{\gamma }}f) & = & 
(f^{\beta }B_{\beta \overline{\gamma }}f),^{\overline{\gamma }}-\frac{n-1}{n}%
(fP_{\beta }f),^{\beta }+\frac{n-1}{8n}fP_{0}f \\ 
& = & e^{\phi }(e^{-\phi }f^{\beta }B_{\beta \overline{\gamma }}f),^{%
\overline{\gamma }}-\frac{n-1}{n}e^{\phi }(e^{-\phi }fP_{\beta }f),^{\beta }+%
\frac{n-1}{8n}fP_{0}f \\ 
&  & -\frac{n-1}{2n}f\langle Pf+\overline{P}f,d_{b}\phi \rangle +\frac{1}{2}%
(f^{\beta }\phi ^{\overline{\gamma }}B_{\beta \overline{\gamma }}f+f^{%
\overline{\beta }}\phi ^{\gamma }B_{\overline{\beta }\gamma }f).%
\end{array}%
\end{equation*}%
By using the identities $\frac{1}{2}\langle \nabla _{b}\phi ,\nabla
_{b}|\nabla _{b}f|^{2}\rangle =f_{\overline{\beta }\gamma }f^{\overline{%
\beta }}\phi ^{\gamma }+f_{\beta \overline{\gamma }}f^{\beta }\phi ^{%
\overline{\gamma }}+f_{\beta \gamma }f^{\beta }\phi ^{\gamma }+f_{\overline{%
\beta }\overline{\gamma }}f^{\overline{\beta }}\phi ^{\overline{\gamma }}$
and%
\begin{equation*}
\begin{array}{c}
f_{\sigma }{}^{\sigma }f_{\overline{\beta }}\phi ^{\overline{\beta }}+f_{%
\overline{\sigma }}{}^{\overline{\sigma }}f_{\beta }\phi ^{\beta }=\frac{1}{2%
}[\Delta _{b}f\left\langle \nabla _{b}f,\nabla _{b}\phi \right\rangle
+nf_{0}\langle J\nabla _{b}f,\nabla _{b}\phi \rangle ],%
\end{array}%
\end{equation*}%
and from (\ref{4a}), we get%
\begin{equation*}
\begin{array}{ll}
& \sum_{\beta ,\gamma }|f_{\beta \overline{\gamma }}-\frac{1}{n}f_{\sigma
}{}^{\sigma }h_{\beta \overline{\gamma }}|^{2}\text{ }=\text{ }(B^{\beta 
\overline{\gamma }}f)(B_{\beta \overline{\gamma }}f) \\ 
= & e^{\phi }(e^{-\phi }f^{\beta }B_{\beta \overline{\gamma }}f),^{\overline{%
\gamma }}-\frac{n-1}{n}e^{\phi }(e^{-\phi }fP_{\beta }f),^{\beta }+\frac{n-1%
}{8n}fP_{0}^{\phi }f+\frac{1}{4}\langle \nabla _{b}\phi ,\nabla _{b}|\nabla
_{b}f|^{2}\rangle \\ 
& -\frac{1}{2}(f_{\beta \gamma }f^{\beta }\phi ^{\gamma }+f_{\overline{\beta 
}\overline{\gamma }}f^{\overline{\beta }}\phi ^{\overline{\gamma }})-\frac{1%
}{4n}[\Delta _{b}f\left\langle \nabla _{b}f,\nabla _{b}\phi \right\rangle
+nf_{0}\langle J\nabla _{b}f,\nabla _{b}\phi \rangle ] \\ 
& +\frac{n-1}{4n}f[\mathcal{L}\left\langle \nabla _{b}f,\nabla _{b}\phi
\right\rangle +n\langle J\nabla _{b}\phi ,\nabla _{b}f_{0}\rangle
+n^{2}f_{0}\phi _{0}].%
\end{array}%
\end{equation*}%
By integrate both sides of the above equation, then apply the equation 
\begin{equation*}
\begin{array}{c}
\int_{M}f[\langle J\nabla _{b}\phi ,\nabla _{b}f_{0}\rangle +nf_{0}\phi
_{0}]d\sigma =\int_{M}f_{0}\langle J\nabla _{b}f,\nabla _{b}\phi \rangle
d\sigma +\frac{1}{2}C_{n}\int_{\Sigma }ff_{0}\phi _{e_{_{n}}}d\Sigma
_{p}^{\phi },%
\end{array}%
\end{equation*}%
use (\ref{40}) to get%
\begin{equation*}
\begin{array}{c}
\sqrt{-1}(P_{n}f-P_{\overline{n}}f)=\sqrt{-1}(P_{n}^{\phi }f-P_{\overline{n}%
}^{\phi }f)-\frac{1}{2}\left\langle \nabla _{b}f,\nabla _{b}\phi
\right\rangle _{e_{2n}}+\frac{n}{2}f_{0}\phi _{e_{n}}%
\end{array}%
\end{equation*}
and (\ref{8}) again, we obtain 
\begin{equation}
\begin{array}{ll}
& \int_{M}\sum_{\beta ,\gamma }|f_{\beta \overline{\gamma }}-\frac{1}{n}%
f_{\sigma }{}^{\sigma }h_{\beta \overline{\gamma }}|^{2}d\sigma \\ 
= & \frac{n-1}{8n}\int_{M}fP_{0}^{\phi }fd\sigma +\frac{n-2}{4}%
\int_{M}(\nabla ^{H})^{2}\phi (\nabla _{b}f,\nabla _{b}f)d\sigma -\frac{n}{8}%
\int_{M}\left( \mathcal{L}\phi \right) |\nabla _{b}f|^{2}d\sigma \\ 
& +\frac{(n-2)(n+1)}{4n}\int_{M}\left( \mathcal{L}f\right) \langle \nabla
_{b}f,\nabla _{b}\phi \rangle d\sigma -\frac{1}{4n}\int_{M}\langle \nabla
_{b}f,\nabla _{b}\phi \rangle ^{2}d\sigma \\ 
& -\frac{1}{2}\int_{M}(f_{\beta \gamma }f^{\beta }\phi ^{\gamma }+f_{%
\overline{\beta }\overline{\gamma }}f^{\overline{\beta }}\phi ^{\overline{%
\gamma }})d\sigma +\frac{1}{4}\sqrt{-1}C_{n}\int_{\Sigma }(f^{\overline{%
\beta }}B_{n\overline{\beta }}f-f^{\beta }B_{\overline{n}\beta }f)d\Sigma
_{p}^{\phi } \\ 
& -\frac{n-1}{4n}\sqrt{-1}C_{n}\int_{\Sigma }f(P_{n}^{\phi }f-P_{\overline{n}%
}^{\phi }f)d\Sigma _{p}^{\phi }+\frac{n-1}{4n}C_{n}\int_{\Sigma
}\left\langle \nabla _{b}f,\nabla _{b}\phi \right\rangle _{e_{2n}}fd\Sigma
_{p}^{\phi } \\ 
& +\frac{n}{16}C_{n}\int_{\Sigma }|\nabla _{b}f|^{2}\phi _{e_{_{2n}}}d\Sigma
_{p}^{\phi }-\frac{n^{2}-n-1}{8n}C_{n}\int_{\Sigma }\left\langle \nabla
_{b}f,\nabla _{b}\phi \right\rangle f_{e_{2n}}d\Sigma _{p}^{\phi }.%
\end{array}
\label{6}
\end{equation}%
Substituting these into the right hand side of (\ref{9}), we final get 
\begin{equation}
\begin{array}{ll}
& \frac{1}{2}\int_{M}\mathcal{L}|\nabla _{b}f|^{2}d\sigma \\ 
= & 2\int_{M}\sum_{\beta ,\gamma }|f_{\beta \gamma }-\frac{1}{2}f_{\beta
}\phi _{\gamma }|^{2}d\sigma -\frac{n+1}{n}\int_{M}(\mathcal{L}f)^{2}d\sigma
+\frac{n+2}{4n}\int_{M}fP_{0}^{\phi }fd\sigma \\ 
& +\int_{M}[Ric+2(\nabla ^{H})^{2}\phi -\frac{n+1}{2}Tor](\nabla
_{b}f,\nabla _{b}f)d\sigma +\frac{n+1}{2}\int_{M}\left( \mathcal{L}f\right)
\langle \nabla _{b}f,\nabla _{b}\phi \rangle d\sigma \\ 
& -\frac{n+2}{4}\int_{M}[\mathcal{L}\phi +\frac{1}{2(n+2)}|\nabla _{b}\phi
|^{2}]|\nabla _{b}f|^{2}d\sigma +\frac{2n+3}{4n}C_{n}\int_{\Sigma }(\mathcal{%
L}f)f_{e_{_{2n}}}d\Sigma _{p}^{\phi } \\ 
& +\frac{1}{4}C_{n}\int_{\Sigma }f_{0}f_{e_{n}}d\Sigma _{p}^{\phi }-\frac{n+2%
}{2n}\sqrt{-1}C_{n}\int_{\Sigma }f(P_{n}^{\phi }f-P_{\overline{n}}^{\phi
}f)d\Sigma _{p}^{\phi } \\ 
& +\frac{1}{2}\sqrt{-1}C_{n}\int_{\Sigma }(f^{\overline{\beta }}B_{n%
\overline{\beta }}f-f^{\beta }B_{\overline{n}\beta }f)d\Sigma _{p}^{\phi }+%
\frac{n+2}{8}C_{n}\int_{\Sigma }|\nabla _{b}f|^{2}\phi _{e_{_{2n}}}d\Sigma
_{p}^{\phi } \\ 
& +\frac{n-1}{2n}C_{n}\int_{\Sigma }\left\langle \nabla _{b}f,\nabla
_{b}\phi \right\rangle _{e_{2n}}fd\Sigma _{p}^{\phi }-\frac{n^{2}+n-1}{4n}%
C_{n}\int_{\Sigma }\left\langle \nabla _{b}f,\nabla _{b}\phi \right\rangle
f_{e_{2n}}d\Sigma _{p}^{\phi }.%
\end{array}
\label{20}
\end{equation}

On the other hand, the divergence theorem (\ref{a}) implies that%
\begin{equation*}
\begin{array}{ll}
& 2\int_{M}\mathcal{L}|\nabla _{b}f|^{2}d\sigma \text{ }=\text{ }%
C_{n}\int_{\Sigma }\left( |\nabla _{b}f|^{2}\right) _{e_{2n}}d\Sigma
_{p}^{\phi } \\ 
= & C_{n}\int_{\Sigma }[\sum_{\beta \neq n}\left( f_{e_{\beta }}f_{e_{\beta
}e_{2n}}+f_{e_{n+\beta }}f_{e_{n+\beta }e_{2n}}\right)
+f_{e_{n}}f_{e_{n}e_{2n}}+f_{e_{2n}}f_{e_{2n}e_{2n}}]d\Sigma _{p}^{\phi }.%
\end{array}%
\end{equation*}%
Substituting the commutation relations $f_{e_{\beta }e_{n+\gamma
}}=f_{e_{n+\gamma }e_{\beta }},\ f_{e_{n+\beta }e_{n+\gamma
}}=f_{e_{n+\gamma }e_{n+\beta }}\ $for all ${\small \beta \neq \gamma }$, $%
f_{e_{n}e_{2n}}=f_{e_{2n}e_{n}}+2f_{0}$ and%
\begin{equation}
\begin{array}{c}
\sum_{\beta \neq n}2(f_{\beta \overline{\beta }}+f_{\overline{\beta }\beta
})+f_{e_{n}e_{n}}=\sum_{j=1}^{2n-1}f_{e_{j}e_{j}}=2\Delta
_{b}^{t}f+H_{p.h}f_{e_{2n}} \\ 
f_{e_{2n}e_{2n}}=2\Delta _{b}f-\sum_{j=1}^{2n-1}f_{e_{j}e_{j}}=2\Delta
_{b}f-2\Delta _{b}^{t}f-H_{p.h}f_{e_{2n}},%
\end{array}
\label{21}
\end{equation}%
into the above equation, then integrating by parts from (\ref{f}) and (\ref%
{g}) yields 
\begin{equation*}
\begin{array}{ll}
& 2\int_{M}\mathcal{L}|\nabla _{b}f|^{2}d\sigma \\ 
= & C_{n}\int_{\Sigma }\sum_{\beta \neq n}(f_{e_{\beta }}f_{e_{2n}e_{\beta
}}+f_{e_{n+\beta }}f_{e_{2n}e_{n+\beta }})d\Sigma _{p}^{\phi
}+C_{n}\int_{\Sigma
}[f_{e_{n}}(f_{e_{2n}e_{n}}+2f_{0})+f_{e_{2n}}f_{e_{2n}e_{2n}}]d\Sigma
_{p}^{\phi } \\ 
= & C_{n}\int_{\Sigma }[\sum_{\beta \neq n}2(f_{\overline{\beta }%
}f_{e_{2n}Z_{\beta }}+f_{\beta }f_{e_{2n}Z\overline{_{\beta }}%
})+f_{e_{n}}f_{e_{2n}e_{n}}]d\Sigma _{p}^{\phi }+C_{n}\int_{\Sigma
}[2f_{e_{n}}f_{0}+f_{e_{2n}}f_{e_{2n}e_{2n}}]d\Sigma _{p}^{\phi } \\ 
= & C_{n}\int_{\Sigma }[\sum_{\beta \neq n}2(f_{\beta }\phi _{\overline{%
\beta }}+f_{\overline{\beta }}\phi _{\beta }-f_{\beta \overline{\beta }}-f_{%
\overline{\beta }\beta })+f_{e_{n}}\phi
_{e_{n}}-f_{e_{n}e_{n}}]f_{e_{2n}}d\Sigma _{p}^{\phi } \\ 
& +C_{n}\int_{\Sigma }[2f_{e_{n}}(f_{0}-\alpha
f_{e_{2n}})+f_{e_{2n}}(2\Delta _{b}f-2\Delta
_{b}^{t}f-H_{p.h}f_{e_{2n}})]d\Sigma _{p}^{\phi } \\ 
& -C_{n}\int_{\Sigma }[f_{e_{n}}(\nabla _{e_{n}}e_{2n})f+f_{e_{2n}}(\nabla
_{e_{n}}e_{n})f+2\sum_{\beta \neq n}(f_{\beta }(\nabla _{Z_{\overline{\beta }%
}}e_{2n})f+f_{\overline{\beta }}(\nabla _{Z_{\beta }}e_{2n})f)]d\Sigma
_{p}^{\phi } \\ 
& +2C_{n}\int_{\Sigma }\sum_{\beta \neq n}[\theta _{\overline{n}}{}^{%
\overline{\beta }}(Z_{\overline{\beta }})f_{n}-\frac{1}{2}\theta _{\overline{%
\beta }}{}^{\overline{n}}(e_{n})f_{\beta }+\theta _{n}{}^{\beta }(Z_{\beta
})f_{\overline{n}}-\frac{1}{2}\theta _{\beta }{}^{n}(e_{n})f_{\overline{%
\beta }}]f_{e_{2n}}d\Sigma _{p}^{\phi } \\ 
= & 2C_{n}\int_{\Sigma }f_{e_{2n}}\left( \Delta _{b}f-2\Delta
_{b}^{t}f\right) d\Sigma _{p}^{\phi }-C_{n}\int_{\Sigma
}H_{p.h}f_{e_{2n}}^{2}d\Sigma _{p}^{\phi }+2C_{n}\int_{\Sigma }[\alpha
f_{e_{n}}f_{e_{2n}}-f_{0}f_{e_{n}}]d\Sigma _{p}^{\phi } \\ 
& +C_{n}\int_{\Sigma }\sum_{j=1}^{2n-1}[\left\langle \nabla
_{e_{j}}e_{n},e_{j}\right\rangle f_{e_{n}}+f_{e_{j}}\phi
_{e_{j}}]f_{e_{2n}}d\Sigma _{p}^{\phi }-C_{n}\int_{\Sigma
}\sum_{j,k=1}^{2n-1}\left\langle \nabla _{e_{j}}e_{2n},e_{k}\right\rangle
f_{e_{j}}f_{e_{k}}d\Sigma _{p}^{\phi },%
\end{array}%
\end{equation*}%
here we use the equations 
\begin{equation*}
\begin{array}{ll}
& 2\sum_{\beta \neq n}[\theta _{\overline{n}}{}^{\overline{\beta }}(Z_{%
\overline{\beta }})f_{n}-\frac{1}{2}\theta _{\overline{\beta }}{}^{\overline{%
n}}(e_{n})f_{\beta }+\theta _{n}{}^{\beta }(Z_{\beta })f_{\overline{n}}-%
\frac{1}{2}\theta _{\beta }{}^{n}(e_{n})f_{\overline{\beta }}] \\ 
= & \sum_{j=1}^{2n-1}\left\langle \nabla _{e_{j}}e_{n},e_{j}\right\rangle
f_{e_{n}}+(\nabla _{e_{n}}e_{n})f+H_{p.h}f_{e_{2n}}%
\end{array}%
\end{equation*}%
and%
\begin{equation*}
\begin{array}{c}
\sum_{\beta \neq n}2[f_{\beta }(\nabla _{Z_{\overline{\beta }}}e_{2n})f+f_{%
\overline{\beta }}(\nabla _{Z_{\beta }}e_{2n})f]+f_{e_{n}}(\nabla
_{e_{n}}e_{2n})f=\sum_{j,k=1}^{2n-1}\left\langle \nabla
_{e_{j}}e_{2n},e_{k}\right\rangle f_{e_{j}}f_{e_{k}},%
\end{array}%
\end{equation*}%
the fact that (\ref{21}) holds only on $\Sigma \backslash S_{\Sigma }.$
However, $d\Sigma _{p}^{\phi }$ can be continuously extends over the
singular set $S_{\Sigma }$ and vanishes on $S_{\Sigma }.$ Finally, by
combining the above integral into (\ref{20}), we can then obtain (\ref{0}).
This completes the proof of the Theorem.
\end{proof}

\section{First Eigenvalue Estimate and Weighted Obata Theorem}

In this section, by applying the weighted CR Reilly formula, we give the
first eigenvalue estimate and derive the Obata-type theorem in a closed
weighted strictly pseudoconvex CR $(2n+1)$-manifold.

\textbf{The Proof of Theorem}\textup{\textbf{\ \ref{Thm}:}}

\begin{proof}
Under the curvature condition $[Ric(\mathcal{L})-\frac{n+1}{2}Tor(\mathcal{L}%
)](\nabla _{b}f_{\mathbb{C}},\nabla _{b}f_{\mathbb{C}})\geq k|\nabla
_{b}f|^{2}$ for a positive constant $k$ and nonnegative weighted CR Paneitz
operator $P_{0}^{\phi },$ the integral formula (\ref{0}) yields%
\begin{equation*}
\begin{array}{ll}
& \frac{n+1}{n}\lambda _{1}^{2}\int_{M}f^{2}d\sigma  \\ 
\geq  & k\int_{M}|\nabla _{b}f|^{2}d\sigma -\frac{n+1}{4}\lambda
_{1}\int_{M}\langle \nabla _{b}f^{2},\nabla _{b}\phi \rangle d\sigma
-(n+2)l\int_{M}|\nabla _{b}f|^{2}d\sigma  \\ 
= & (k-(n+2)l)\lambda _{1}\int_{M}f^{2}d\sigma +\frac{n+1}{4}\lambda
_{1}\int_{M}\phi \mathcal{L}f^{2}d\sigma  \\ 
= & (k-(n+2)l)\lambda _{1}\int_{M}f^{2}d\sigma +\frac{n+1}{2}\lambda
_{1}\int_{M}\phi \lbrack |\nabla _{b}f|^{2}-\lambda _{1}f^{2}]d\sigma  \\ 
\geq  & (k-(n+2)l)\lambda _{1}\int_{M}f^{2}d\sigma +\frac{n+1}{2}\lambda
_{1}\int_{M}[(\inf \phi )|\nabla _{b}f|^{2}-(\sup \phi )\lambda
_{1}f^{2}]d\sigma  \\ 
= & [(k-(n+2)l)\lambda _{1}-\frac{n+1}{2}\omega \lambda
_{1}^{2}]\int_{M}f^{2}d\sigma ,%
\end{array}%
\end{equation*}%
here we assume $\mathcal{L}\phi +\frac{1}{2(n+2)}|\nabla _{b}\phi |^{2}\leq
4l$ for some nonnegative constant $l$ and let $\omega =\underset{M}{\mathrm{%
osc}}\phi =\underset{M}{\sup }\phi -\underset{M}{\inf }\phi $. It implies\
that the first eigenvalue $\lambda _{1}$ will satisfies%
\begin{equation*}
\begin{array}{c}
\lambda _{1}\geq \frac{2n[k-(n+2)l]}{(n+1)(2+n\omega )}.%
\end{array}%
\end{equation*}

Moreover, if the above inequality becomes equality, then $\omega =\underset{M%
}{\mathrm{osc}}\phi =0$ and thus the weighted function $\phi $ will be
constant, we have $\mathcal{L=}$ $\Delta _{b}$ and $P_{0}^{\phi }=P_{0}$. In
this case we can let $l=0,$ then the first eigenvalue of the sub-Laplacian $%
\Delta _{b}$ achieves the sharp lower bound 
\begin{equation*}
\begin{array}{c}
\lambda _{1}=\frac{nk}{n+1},%
\end{array}%
\end{equation*}%
and it reduces to the original Obata-type Theorem for the sub-Laplacian $%
\Delta _{b}$ in a closed strictly pseudoconvex CR $(2n+1)$-manifold $%
(M,J,\theta ).$ It following from Chang-Chiu \cite{cc1}, \cite{cc2} and
Li-Wang \cite{lw} that $M$\ is CR isometric to a standard CR $(2n+1)$-sphere.
\end{proof}

\section{First Dirichlet Eigenvalue Estimate and Weighted Obata Theorem}

In this section, we derive the first Dirichlet eigenvalue estimate in a
compact weighted strictly pseudoconvex CR $(2n+1)$-manifold $(M,J,\theta
,d\sigma )$ with boundary $\Sigma $ and its corresponding weighted
Obata-type Theorem.

\textbf{The Proof of Theorem}\textup{\textbf{\ \ref{TB}:}}

\begin{proof}
Since $f=0$ on $\Sigma $ and $e_{j}$ is tangent along $\Sigma $ for $1\leq
j\leq 2n-1$, then $f_{e_{j}}=0$ for $1\leq j\leq 2n-1$ and $\Delta _{b}^{t}f=%
\frac{1}{2}\sum_{j=1}^{2n-1}[\left( e_{j}\right) ^{2}-(\nabla
_{e_{j}}e_{j})^{t}]f=0$ on $\Sigma .$ Furthermore, since $\mathcal{L}%
f=-\lambda _{1}f\ $on$\mathrm{\ }M$ and $f=0$ on $\Sigma ,$ then $\mathcal{L}%
f=0$ on $\Sigma $. It follows from (\ref{21}) and $\Delta _{b}f=\mathcal{L}%
f+\langle \nabla _{b}f,\nabla _{b}\phi \rangle $ that 
\begin{equation*}
\begin{array}{ll}
& 4\sqrt{-1}C_{n}\int_{\Sigma }(f^{\overline{\beta }}B_{n\overline{\beta }%
}f-f^{\beta }B_{\overline{n}\beta }f)d\Sigma _{p}^{\phi } \\ 
= & C_{n}\int_{\Sigma }\sum_{\beta \neq n}[f_{e_{\beta }}(f_{e_{\beta
}e_{2n}}-f_{e_{n+\beta }e_{n}})+f_{e_{n+\beta }}(f_{e_{\beta
}e_{n}}+f_{e_{n+\beta }e_{2n}})]d\Sigma _{p}^{\phi } \\ 
& +C_{n}\int_{\Sigma }f_{e_{2n}}[(f_{e_{n}e_{n}}+f_{e_{2n}e_{2n}})-\frac{2}{n%
}\Delta _{b}f]d\Sigma _{p}^{\phi } \\ 
= & C_{n}\int_{\Sigma }f_{e_{2n}}\{[\left( e_{n}\right) ^{2}-(\nabla
_{e_{n}}{}^{e_{n}})]f+(\frac{2n-2}{n}\Delta _{b}f-2\Delta
_{b}^{t}f-H_{p.h}f_{e_{2n}})\}d\Sigma _{p}^{\phi } \\ 
= & \frac{n-1}{n}C_{n}\int_{\Sigma }\phi _{e_{2n}}f_{e_{2n}}^{2}d\Sigma
_{p}^{\phi }-C_{n}\int_{\Sigma }(H_{p.h}+\tilde{\omega}_{n}^{%
\;n}(e_{n}))f_{e_{2n}}^{2}d\Sigma _{p}^{\phi }.%
\end{array}%
\end{equation*}%
Also under the curvature condition $[Ric(\mathcal{L})-\frac{n+1}{2}Tor(%
\mathcal{L})](\nabla _{b}f_{\mathbb{C}},\nabla _{b}f_{\mathbb{C}})\geq
k|\nabla _{b}f|^{2}$ for a positive constant $k,$ the weighted CR Paneitz
operator $P_{0}^{\phi }$ is nonnegative and $H_{p.h}-\tilde{\omega}%
_{n}^{\;n}(e_{n})-\frac{n+2}{2}\phi _{e_{2n}}\geq 0$ on $\Sigma $, then the
integral formula (\ref{0}) yields%
\begin{equation*}
\begin{array}{ll}
& \frac{n+1}{n}\lambda _{1}^{2}\int_{M}f^{2}d\sigma \\ 
\geq & k\int_{M}|\nabla _{b}f|^{2}d\sigma -\frac{n+1}{4}\lambda
_{1}\int_{M}\langle \nabla _{b}f^{2},\nabla _{b}\phi \rangle d\sigma
-(n+2)l\int_{M}|\nabla _{b}f|^{2}d\sigma \\ 
& +\frac{1}{8}C_{n}\int_{\Sigma }[H_{p.h}-\tilde{\omega}_{n}^{\;n}(e_{n})-%
\frac{n+2}{2}\phi _{e_{2n}}]f_{e_{2n}}^{2}d\Sigma _{p}^{\phi } \\ 
= & [k-(n+2)l]\lambda _{1}\int_{M}f^{2}d\sigma +\frac{n+1}{4}\lambda
_{1}\int_{M}\phi \mathcal{L}f^{2}d\sigma \\ 
= & [k-(n+2)l]\lambda _{1}\int_{M}f^{2}d\sigma +\frac{n+1}{2}\lambda
_{1}\int_{M}\phi \lbrack |\nabla _{b}f|^{2}-\lambda _{1}f^{2}]d\sigma \\ 
\geq & [(k-(n+2)l)\lambda _{1}-\frac{n+1}{2}\omega \lambda
_{1}^{2}]\int_{M}f^{2}d\sigma ,%
\end{array}%
\end{equation*}%
here we assume $\mathcal{L}\phi +\frac{1}{2(n+2)}|\nabla _{b}\phi |^{2}\leq
4l$ on $M$ for some nonnegative constant $l$ and $\omega =\underset{M}{%
\mathrm{osc}}\phi =\underset{M}{\sup }\phi -\underset{M}{\inf }\phi $. It
implies\ that the first eigenvalue of $\mathcal{L}$ will satisfy%
\begin{equation*}
\begin{array}{c}
\lambda _{1}\geq \frac{2n[k-(n+2)l]}{(n+1)(2+n\omega )}.%
\end{array}%
\end{equation*}

Moreover, if the above inequality becomes equality, then $\omega =\underset{M%
}{\mathrm{osc}}\phi =0$ and thus the weighted function $\phi $ will be
constant, we get $\mathcal{L=}$ $\Delta _{b}$ and $P_{0}^{\phi }=P_{0},$
which is nonnegative for $n\geq 2$ by Lemma \ref{lemma 3.1}. Note that we
need to assume that $P_{0}$ is nonnegative for $n=1$. Also the corresponding
eigenfunction $f$ will satisfy 
\begin{equation*}
\begin{array}{c}
f_{\alpha \beta }=0\text{ \ \textrm{for all} }\alpha ,\beta ,%
\end{array}%
\end{equation*}%
\begin{equation*}
\begin{array}{c}
\lbrack Ric-\frac{n+1}{2}Tor](\nabla _{b}f_{\mathbb{C}},\nabla _{b}f_{%
\mathbb{C}})=k|\nabla _{b}f|^{2},%
\end{array}%
\end{equation*}%
and 
\begin{equation*}
\begin{array}{c}
P_{0}f=0%
\end{array}%
\end{equation*}%
on $M,$ and $f=0$ on $\Sigma .$ We also have $H_{p.h}=0$ and $\tilde{\omega}%
_{n}^{\;n}(e_{n})=0$ on $\Sigma .$ In this case we can let $l=0,$ then the
first eigenvalue of the sub-Laplacian $\Delta _{b}$ achieves the sharp lower
bound 
\begin{equation*}
\begin{array}{c}
\lambda _{1}=\frac{nk}{n+1}.%
\end{array}%
\end{equation*}

By applying the same method of Li-Wang \cite{lw}, it can be showed that the
pseudohermitian torsion vanishes on $M$ with boundary $\Sigma $. Then we
follow from Chang-Chiu \cite{cc2} to define the Webster (adapted) Riemannian
metric $g_{\varepsilon }$ of $(M,J,\theta )$ by 
\begin{equation*}
\begin{array}{c}
g_{\varepsilon }=\varepsilon ^{2}\theta ^{2}+\frac{1}{2}d\theta (\cdot
,J\cdot )\text{ }\mathrm{for}\text{ }\varepsilon >0\text{ }\mathrm{with}%
\text{ }(n+1)\varepsilon ^{2}=k.%
\end{array}%
\end{equation*}%
Since $Ric(Z,Z)\geq k\left\langle Z,Z\right\rangle $ and free torsion, the
Theorem 4.9 in \cite{cc2} says that the Ricci curvature of $g_{\varepsilon }$
satisfies 
\begin{equation*}
\begin{array}{c}
Rc_{g_{\varepsilon }}\geq \lbrack (2n+1)-1]\frac{k}{n+1}.%
\end{array}%
\end{equation*}%
And the mean curvature $H_{\varepsilon }$ is zero on $\Sigma $ with respect
to the metric $g_{\varepsilon },$ which will be presented in the next. Then
Theorem 1.2 in \cite{cc2} will yield that the eigenfunction $f$ of $\Delta
_{b}$\ achieves the sharp lower bound for the first Dirichlet eigenvalue of
the Laplace $\Delta _{\varepsilon }$ with respect to the Riemannian metric $%
g_{\varepsilon }$. Therefore, by Theorem $4$ in \cite{Re}, $M$ is isometric
to a hemisphere in a standard CR $(2n+1)$-sphere.

In the following we show that the mean curvature $H_{\varepsilon }$ is zero
on $\Sigma $ with respect to the Webster metric $g_{\varepsilon }$. We
choose $\{l_{j}=e_{j},l_{2n}=\frac{\alpha e_{2n}+T}{\sqrt{\alpha
^{2}+\varepsilon ^{2}}}\}_{j=1}^{2n-1},$ here $e_{n+\beta }=Je_{\beta }$ for 
$1\leq \beta \leq n,$ to form an orthonormal tangent frame and $\nu =\frac{1%
}{\sqrt{\alpha ^{2}+\varepsilon ^{2}}}\left( \varepsilon e_{2n}-\frac{\alpha 
}{\varepsilon }T\right) $ be an unit normal vector on the non-singular set $%
\Sigma \backslash S_{\Sigma }.$ The second fundamental form is then defined
by $h_{ij}^{\varepsilon }=-\left\langle \nu ,\nabla
_{l_{i}}^{R}l_{j}\right\rangle _{g_{\varepsilon }}$ and the mean curvature
is defined by $H_{\varepsilon }=\sum_{j=1}^{2n}h_{jj}^{\varepsilon },$ here $%
\nabla ^{R}$ is the corresponding Riemannian connection. By the Lemma 4.3 in 
\cite{cc2}, we have%
\begin{equation*}
\begin{array}{lll}
\nabla ^{R}e_{\beta } & = & \omega _{\beta }^{\;\gamma }\otimes e_{\gamma }+(%
\tilde{\omega}_{\beta }^{\;\gamma }+\varepsilon ^{2}\delta _{\beta \gamma
}\theta )\otimes e_{n+\gamma }+e^{n+\beta }\otimes T, \\ 
\nabla ^{R}e_{n+\beta } & = & -(\tilde{\omega}_{\beta }^{\;\gamma
}+\varepsilon ^{2}\delta _{\beta \gamma }\theta )\otimes e_{\gamma }+\omega
_{\beta }^{\;\gamma }\otimes e_{n+\gamma }-e^{\beta }\otimes T, \\ 
\nabla ^{R}T & = & -\varepsilon e^{n+\gamma }\otimes e_{\gamma }+\varepsilon
e^{\gamma }\otimes e_{n+\gamma },%
\end{array}%
\end{equation*}%
for $\beta =1,2,\cdots ,n.$ Then the mean curvature $H_{\varepsilon }$ is
given explicitly by 
\begin{equation*}
\begin{array}{lll}
H_{\varepsilon } & = & -\sum_{j=1}^{2n-1}\left\langle \nu ,\nabla
_{l_{j}}^{R}l_{j}\right\rangle _{g_{\varepsilon }}-\left\langle \nu ,\nabla
_{l_{2n}}^{R}l_{2n}\right\rangle _{g_{\varepsilon }} \\ 
& = & -\sum_{j=1}^{2n-1}\frac{1}{\sqrt{\alpha ^{2}+\varepsilon ^{2}}}%
\left\langle \varepsilon e_{2n}-\frac{\alpha }{\varepsilon }T,\nabla
_{e_{j}}e_{j}\right\rangle _{g_{\varepsilon }} \\ 
&  & -\frac{1}{\left( \alpha ^{2}+\varepsilon ^{2}\right) ^{2}}\left\langle
\varepsilon e_{2n}-\frac{\alpha }{\varepsilon }T,l_{2n}(\alpha )\left(
\varepsilon ^{2}e_{2n}-\alpha T\right) \right\rangle _{g_{\varepsilon }} \\ 
& = & \frac{\varepsilon }{\sqrt{\alpha ^{2}+\varepsilon ^{2}}}H_{p.h}-\frac{%
\varepsilon }{\alpha ^{2}+\varepsilon ^{2}}l_{2n}(\alpha ),%
\end{array}%
\end{equation*}%
and thus 
\begin{equation*}
H_{\varepsilon }=H_{p.h}=0
\end{equation*}%
on $\Sigma ,$ if $\alpha $ vanishes on $\Sigma .$

In the following we show that $\alpha $ vanishes identically on $\Sigma .$
It follows from the formula $\left[ \Delta _{b},T\right] f=4{\func{Im}}[%
\sqrt{-1}\left( A^{\beta \gamma }f_{\beta }\right) ,_{\gamma }]=0$ that $%
\Delta _{b}f_{0}=-\lambda _{1}f_{0}.$ Again from the equation 
\begin{equation*}
\begin{array}{c}
\int_{M}\psi _{0}d\mu =\int_{M}\mathrm{div}_{b}\left( J\nabla _{b}\psi
\right) d\mu =\frac{1}{2}C_{n}\int_{\Sigma }\psi _{e_{n}}d\Sigma _{p}=0%
\end{array}%
\end{equation*}%
for any real function $\psi $ which vanishes on $\Sigma $ to get that $%
\int_{M}\left( f^{2}\right) _{0}d\mu =0$ for $f=0$ on $\Sigma $ and also 
\begin{equation*}
\begin{array}{c}
\int_{M}[f\Delta _{b}f_{0}+\left\langle \nabla _{b}f,\nabla
_{b}f_{0}\right\rangle ]d\mu =\frac{1}{2}C_{n}\int_{\Sigma
}ff_{0e_{2n}}d\Sigma _{p}=0,%
\end{array}%
\end{equation*}%
which implies that $\int_{M}\left\langle \nabla _{b}f,\nabla
_{b}f_{0}\right\rangle d\mu =0$. Thus%
\begin{equation}
\begin{array}{c}
0=\int_{M}[f_{0}\Delta _{b}f+\left\langle \nabla _{b}f,\nabla
_{b}f_{0}\right\rangle ]d\mu =\frac{1}{2}C_{n}\int_{\Sigma
}f_{0}f_{e_{2n}}d\Sigma _{p}.%
\end{array}
\label{22}
\end{equation}%
On the other hand, since $f=0$ on $\Sigma $ and $\alpha e_{2n}+T$ is tangent
along $\Sigma $ from the definition of $\alpha ,$ so on $\Sigma $ we have%
\begin{equation*}
\begin{array}{c}
f_{0}=-\alpha f_{e_{2n}}%
\end{array}%
\end{equation*}%
and 
\begin{equation*}
\begin{array}{c}
-\alpha \theta \wedge e^{n}\wedge \left( d\theta \right) ^{n-1}=e^{n}\wedge
e^{2n}\wedge \left( d\theta \right) ^{n-1}=\frac{1}{2n}\left( d\theta
\right) ^{n}%
\end{array}%
\end{equation*}%
which is a nonnegative $2n$-form on $\Sigma $. It follows from (\ref{22})
that 
\begin{equation*}
\begin{array}{c}
0=\frac{1}{2}C_{n}\int_{\Sigma }f_{0}f_{e_{2n}}d\Sigma _{p}=-\int_{\Sigma
}\alpha f_{e_{2n}}^{2}\theta \wedge e^{n}\wedge \left( d\theta \right)
^{n-1}=\frac{1}{2n}\int_{\Sigma }f_{e_{2n}}^{2}\left( d\theta \right) ^{n}.%
\end{array}%
\end{equation*}%
This will yield that $\alpha f_{e_{2n}}^{2}=0$ on $\Sigma $ and therefore $%
f_{0}=0$ on $\Sigma ,$ so $T$ is tangent to $\Sigma $ and thus $\alpha $
vanishes identically on $\Sigma \backslash S_{\Sigma }$ and continuously
extends over the singular set $S_{\Sigma }$ and the same constant on $%
S_{\Sigma }$.
\end{proof}

\end{document}